\documentclass{amsart}

\usepackage{amsmath,amssymb,amsthm,graphicx,pstricks,amscd,amsfonts,latexsym}
\usepackage{xcolor}
\usepackage[all]{xy}
\usepackage[normalem]{ulem}
\usepackage{comment}
\usepackage{lscape}
\usepackage{hyperref}
\usepackage{tikz}

\usetikzlibrary{shapes}
\usetikzlibrary{decorations.pathreplacing,decorations.markings} 

\usepackage{graphicx}
\usepackage{caption} 
\usepackage{float}

\usepackage{tikz,tikz-cd}
\usetikzlibrary{snakes}
\usetikzlibrary{shapes.geometric,positioning}
\usetikzlibrary{arrows,decorations.pathmorphing,decorations.pathreplacing}

\usepackage{todonotes}

\theoremstyle{plain} 
\newtheorem{theorem}{Theorem}[section]
\newtheorem{lemma}[theorem]{Lemma}
\newtheorem{proposition}[theorem]{Proposition}
\newtheorem{corollary}[theorem]{Corollary}
\newtheorem{conjecture}[theorem]{Conjecture}

\theoremstyle{definition} 
\newtheorem{example}[theorem]{Example}
\newtheorem{examples}[theorem]{Examples}
\newtheorem{definition}[theorem]{Definition}
\newtheorem{remark}[theorem]{Remark}

\newcommand{\Coker}{\mbox{\rm Coker\,}}

\newcommand{\End}[1]{\operatorname{\rm End}_{#1}}

\newcommand{\Hom}[1]{\operatorname{{\rm Hom}}_{#1}}

\newcommand{\RHom}[1]{\operatorname{{\rm RHom}}_{#1}}
\newcommand{\Ext}[2]{\operatorname{\rm Ext}^{#1}_{#2}}

\newcommand{\rad}{\operatorname{rad}}

\newcommand{\add}{\mbox{{\rm add \!}}}
\newcommand{\codim}{\operatorname{codim}}
\newcommand{\GL}{\mbox{{\rm GL \!\!}}}
\newcommand{\MOD}{\operatorname{mod}}

\renewcommand{\mod}{\mbox{{\rm mod \!}}}
\newcommand{\rep}{\mbox{{\rm rep \!\!}}}

\newcommand{\proj}{\operatorname{{\rm proj }}}

\newcommand{\Ho}{\operatorname{H}}

\newcommand{\ltensor}[1]{\underset{#1}{\overset{\rm L}{\otimes}}}

\newcommand{\iso}{\xrightarrow{_\sim}}

\newcommand{\cA}{\mathcal{A}}

\newcommand{\cC}{\mathcal{C}}
\newcommand{\cD}{\mathcal{D}}

\newcommand{\cF}{\mathcal{F}}

\newcommand{\cM}{\mathcal{M}}

\newcommand{\cO}{\mathcal{O}}

\newcommand{\cS}{\mathcal{S}}
\newcommand{\cT}{\mathcal{T}}
\newcommand{\cU}{\mathcal{U}}

\newcommand{\cW}{\mathcal{W}}
\newcommand{\cX}{\mathcal{X}}

\newcommand{\cZ}{\mathcal{Z}}
\newcommand{\bA}{\mathbb{A}}
\newcommand{\bC}{\mathbb{C}}

\newcommand{\bQ}{\mathbb{Q}}
\newcommand{\bR}{\mathbb{R}}

\newcommand{\bZ}{\mathbb{Z}}
\newcommand{\bd}{\mathbf{d}}
\newcommand{\be}{\mathbf{e}}
\newcommand{\bg}{\mathbf{g}}
\newcommand{\bh}{\mathbf{h}}

\definecolor{dark-green}{RGB}{14,150,2}

\newcommand{\op}{\operatorname{op}}

\newcommand{\cyl}{\operatorname{Cyl}}
\newcommand{\Wall}{\operatorname{Wall}}
\newcommand{\prin}{\operatorname{prin}}

\newcommand{\cluster}{\operatorname{cluster}}

\newcommand{\Supp}{\mbox{{\rm Supp}}}
\newcommand{\fD}{\mathfrak{D}}

\newcommand{\fg}{\mathfrak{g}}
\newcommand{\fM}{\mathfrak{M}}
\newcommand{\fS}{\mathfrak{S}}
\newcommand{\cI}{\mathcal{I}}
\newcommand{\tsilt}{\operatorname{2-silt}}

\makeatletter
\let\@wraptoccontribs\wraptoccontribs
\makeatother
\newcommand{\id}{\mathbf{1}}
\newcommand{\ten}{\otimes}

\begin{document}
\title[]{Tame algebras have dense $\bg$-vector fans}

\author{Pierre-Guy Plamondon}
\address{P.-G.P.~: Universit\'e Paris-Saclay, CNRS, Laboratoire de math\'ematiques d'Orsay, 91405, Orsay, France}
\email{pierre-guy.plamondon@universite-paris-saclay.fr}
\thanks{The first author was supported by French ANR grant CHARMS (ANR-19-CE40-0017-02).}

\author{Toshiya Yurikusa}
\address{T.Y.~: Mathematical Institute, Tohoku University, Aoba-ku, Sendai, 980-8578, Japan}
\email{toshiya.yurikusa.d8@tohoku.ac.jp}
\thanks{The second author was supported by JSPS KAKENHI Grant Number 20J00410.}

\address{B.~K.~: Universit\'e Paris de Paris\\
    UFR de Math\'ematiques\\
    CNRS\\
   Institut de Math\'ematiques de Jussieu--Paris Rive Gauche, IMJ-PRG \\   
    B\^{a}timent Sophie Germain\\
    75205 Paris Cedex 13\\
    France
}
\email{bernhard.keller@imj-prg.fr}
\urladdr{https://webusers.imj-prg.fr/~bernhard.keller/}

\keywords{}
\thanks{
}

\date{\today}

\subjclass[2010]{
}

\contrib[with an appendix by]{Bernhard Keller}

\begin{abstract}
 The~$\bg$-vector fan of a finite-dimensional algebra is a fan whose rays are the~$\bg$-vectors of its~$2$-term presilting objects.  We prove that the~$\bg$-vector fan of a tame algebra is dense.  We then apply this result to obtain a near classification of quivers for which the closure of the cluster~$\bg$-vector fan is dense or is a half-space, using the additive categorification of cluster algebras by means of Jacobian algebras.  As another application, we prove that for quivers with potentials arising from once-punctured closed surfaces, the stability and cluster scattering diagrams only differ by wall-crossing functions on the walls contained in a separating hyperplane. 
The appendix is devoted to the construction of truncated twist functors
and their adjoints.
\end{abstract}

\maketitle

\tableofcontents

\section{Introduction}

In the theory of cluster algebras,~$\bg$-vectors (introduced in~\cite{FominZelevinsky2007}) play an important role: they endow the cluster algebra with principal coefficients with a grading, and are computationally lighter than cluster variables.

The~$\bg$-vectors of a cluster algebra are the rays of a fan~$\cF^{\bg}_{\cluster}$ whose cones are generated by compatible~$\bg$-vectors.  As a consequence of results of~\cite{DerksenWeymanZelevinsky2}, this fan is essential, rational, polyhedral and simplicial.  The fan of~$\bg$-vectors plays an important role in the theory and appears in different guises in the literature: it appears in the tropical cluster~$\cX$-variety (or cluster Poisson variety) of~\cite{FockGoncharov}; it is a subfan of the underlying fan of the cluster scattering diagram of~\cite{GHKK} and of the stability scattering diagram of~\cite{bridgeland2017}; for cluster algebras of finite type, it is a complete fan~\cite{FominZelevinsky2003}, and is the normal fan of a polytope called the generalized associahedron~\cite{ChapotonFominZelevinsky,HohlwegPilaudStella,ArkaniHamedBaiHeYan,BrustleDouvilleMousavandThomasYildirim,PadrolPaluPilaudPlamondon}.

A related notion in representation theory is the~$\bg$-vector fan~$\cF^{\bg}_{\tsilt}$ of~$2$-term silting complexes over a finite-dimensional algebra~$\Lambda$.  It follows from the work of~\cite{AdachiIyamaReiten} that the fan~$\cF^{\bg}_{\tsilt}$ is an essential, rational, polyhedral, simplicial fan, and from~\cite{DemonetIyamaJasso} that it is complete if and only if the algebra~$\Lambda$ is~$\tau$-tilting finite.

In this paper, we are interested in the case where the~$\bg$-vector fans are not complete, but dense in the real space~$\bR^n$ which contains them.  An algebra such that~$\cF^{\bg}_{\tsilt}$ is dense was called \emph{$\bg$-tame} in~\cite{AokiYurikusa}.  Our main result is the following.

\begin{theorem}[Theorem \ref{theo::density}]\label{theo::mainIntro}
 Let~$\Lambda$ be a finite-dimensional basic algebra over an algebraically closed field.  If~$\Lambda$ is tame, then its fan~$\cF^{\bg}_{\tsilt}$ is dense in~$\bR^n$. 
 In other words, if~$\Lambda$ is tame, then~$\Lambda$ is~$\bg$-tame.
\end{theorem}
Theorem~\ref{theo::mainIntro} was previously known for path algebras of extended Dynkin type by the results of~\cite{Hille}; for representation-finite algebras by~\cite{DemonetIyamaJasso} (and in fact for any~$\tau$-tilting finite algebra, even if it is not tame); for Jacobian algebras associated to triangulated surfaces by~\cite{Yurikusa} (tame or not); for gentle algebras by~\cite{AokiYurikusa}; and for special biserial algebras by~\cite{AokiYurikusa} and~\cite{AsaiDemonetIyama}, both using different methods.

We note that the converse of Theorem~\ref{theo::mainIntro} is false: the quiver associated with a triangulation of a torus with one boundary component and one marked point on it admits a wild potential by~\cite{GeissLabardiniFragosoSchroer}, even though its~$\bg$-vector fan is dense by~\cite{Yurikusa}.  The class of~$\bg$-tame algebras is thus strictly larger than that of tame algebras. 

The related notion of~$\tau$-tilting tameness was introduced in~\cite{BrustleSmithTreffinger}: an algebra is~\emph{$\tau$-tilting tame} if the closure~$\overline{\Wall}$ of the union of the walls containing non-trivial stability conditions in~$\bR^n$ has measure zero.  Using recent results, it is not hard to see that~$\tau$-tilting tameness implies~$\bg$-tameness: by~\cite[Theorem 3.17]{Asai},~$\overline{\Wall}$ is the complement of the union of the interior of the maximal cones of~$\cF^{\bg}_{\tsilt}$, so if the latter is not dense,~$\overline{\Wall}$ contains the cone generated by an open subset of~$\bR^n$, and so has infinite measure.  We conjecture that~$\tau$-tilting tameness is equivalent to~$\bg$-tameness.  This, combined with Theorem~\ref{theo::mainIntro}, would imply that tame algebras are~$\tau$-tilting tame, as conjectured in~\cite[Conjecture 3.22]{BrustleSmithTreffinger}.  We do not discuss~$\tau$-tilting tameness further in this paper.

\smallskip

We provide two applications of Theorem~\ref{theo::mainIntro}.  The first is a near classification of quivers~$Q$ whose cluster algebras have dense~$\bg$-vector fans.
We say that a quiver~$Q$ is \emph{cluster-$\bg$-dense} in that case, and~\emph{half cluster-$\bg$-dense} if the closure of both its cluster fan and that of~$Q^{op}$ are half-spaces.
\begin{corollary}[Theorem \ref{theo::classify g-dense cluster}] \label{coro::mutationFiniteIntro}
 Let~$Q$ be a quiver without loops and $2$-cycles.  Assume that~$Q$ is not mutation equivalent to~$X_6, X_7$ (see Theorem~\ref{theo::classification of mutation finite}) or~$K_m$, with~$m\geq 3$ (see Example~\ref{ex::rank2}).  Then~$Q$ is mutation-finite if and only if it is cluster-$\bg$-dense or half cluster-$\bg$-dense.  In this case,
 \begin{itemize}
  \item $Q$ is half cluster-$\bg$-dense if it arises from a triangulation of a closed surface with exactly one puncture, and
  \item $Q$ is cluster-$\bg$-dense otherwise.
 \end{itemize}

\end{corollary}
We conjecture that~$X_6$ and~$X_7$ behave no differently.

\begin{conjecture}
 \begin{enumerate}
  \item The~$\bg$-vector fan in type~$X_6$ is dense.
  \item The closure of the~$\bg$-vector fan in type~$X_7$ is a half-space.
 \end{enumerate}
\label{conj::intro}
\end{conjecture}

We prove Corollary~\ref{coro::mutationFiniteIntro} by applying Theorem~\ref{theo::mainIntro} and a result of~\cite{GeissLabardiniFragosoSchroer} that states that all mutation-finite quivers admit a tame non-degenerate potential, except for~$X_6$, $X_7$ and~$K_m$, with~$m\geq 3$.  We note that Corollary~\ref{coro::mutationFiniteIntro} was already known for quivers of Dynkin type~\cite{FominZelevinsky2003}, extended Dynkin type~\cite{Hille} and of surface type~\cite{Yurikusa}; thus the novelty here is the inclusion of the quivers of type~$E_6^{(1,1)}, E_7^{(1,1)}$ and~$E_8^{(1,1)}$.

\smallskip

The second application is a consequence of an argument due to L.~Mou~\cite{Mou} to get information on a conjecture of~\cite{KontsevichSoibelman}.

\begin{theorem}[Theorem \ref{theo::scatteringdiags}]\label{theo::scatteringIntro}
Let~$Q$ be a quiver without loops and~$2$-cycles, and let~$W$ be a non-degenerate potential on it.  Let~$\fD_Q$ be the cluster scattering diagram of~$Q$, and let~$\fD_{J(Q,W)}$ be the stability scattering diagram of the Jacobian algebra~$J(Q,W)$ (see Section~\ref{sect::scattering}).
 \begin{enumerate} 
\item If $Q$ is cluster-$\bg$-dense, then the cluster scattering diagram $\fD_Q$ is equal to the stability scattering diagram $\fD_{J(Q,W)}$.
\item If $Q$ is half cluster-$\bg$-dense, then $\fD_Q$ and $\fD_{J(Q,W)}$ only differ by functions on walls in the separating hyperplane.
\end{enumerate}
\end{theorem}

This allows to obtain new information on the scattering diagrams of
cluster algebras from once-punctured closed surfaces.

\begin{corollary}\label{coro::scatteringIntro}
 For a cluster algebra arising from a once punctured closed
surface, the two scattering diagrams only differ by functions on the
walls of the separating hyperplane.
\end{corollary}
We note that Conjecture~\ref{conj::intro}, together with Theorem~\ref{theo::scatteringIntro}, would imply similar results in types~$X_6$ and~$X_7$.
Cases where the two scattering diagrams are known to differ by central elements are Jacobian algebras admitting a green-to-red sequence and the Jacobian algebra of a once-punctured torus~\cite{Mou}, and examples arising from del Pezzo surfaces~\cite{BeaujardManschotPioline}; the latter reference also contains example of scattering diagrams for which the property does not hold, as explained for example in the lecture~\cite{Bousseau} by Pierrick Bousseau.

The proof of Theorem~\ref{theo::mainIntro} relies on two main ingredients.  On one hand, it uses in an essential way the results on generic decompositions of~$\bg$-vectors obtained in~\cite{DerksenFei,plamondon2013} and their more precise formulation for tame algebras shown in~\cite{GeissLabardiniFragosoSchroer2} (see Theorem~\ref{theo::genericDecompositionForTame}).  On the other hand, it relies on an operation~$\cyl_X$ on the objects of the homotopy category~$K^b(\proj\Lambda)$.  The operation~$\cyl_X$ is reminiscent of a spherical twist with respect to an object~$X$~\cite{SeidelThomas}, with the important caveat that~$\cyl_X$ is not an auto-equivalence; in fact, it is not even a functor.  The main difficulty in our proof of Theorem~\ref{theo::mainIntro} is to show that~$\cyl_X$, despite not being as good as a twist functor, behaves nicely enough on a subcategory of~$K^b(\proj\Lambda)$.  

Our strategy is similar to (and partly inspired by) the one used in~\cite{AokiYurikusa}, where Theorem~\ref{theo::mainIntro} is proved for completed gentle algebras by using the geometric model of~\cite{OpperPlamondonSchroll}.  The idea there is to notice that the boundary of the~$\bg$-vector fan is governed by the~$\bg$-vectors of band objects, which correspond to closed curves in the geometric model, and to approach this boundary by inflicting successive Dehn twists on arcs with respect to those closed curves.  Our operators~$\cyl_X$ should be seen as an ``algebraic counterpart'' of these Dehn twists, and indeed the two notions coincide in the case of gentle algebras.
Under suitable vanishing conditions, the action of the operator $\cyl_X$ coincides
with that of a truncated dual twist functor as constructed in the appendix.

This paper is organized as follows.  In Section~\ref{sect::recollections}, we collect definitions and results on generic decomposition of~$\bg$-vectors and varieties of representations of an algebra which will be needed in the paper.  In Section~\ref{sect::gVectorFansAndTameAlgebras}, we define the~$\bg$-vector fan of~$2$-term silting objects of an algebra, and we state a result of~\cite{GeissLabardiniFragosoSchroer2} on the generic decomposition of~$\bg$-vectors for tame algebras (Theorem~\ref{theo::genericDecompositionForTame}).  In Section~\ref{sect::densityOfFan}, we introduce the operator~$\cyl_X$ and prove Theorem~\ref{theo::mainIntro}.  We then apply our results to the density of~$\bg$-vector fans of cluster algebras in Section~\ref{sect::densitiForMutationFinite} and to scattering diagrams in Section~\ref{sect::scattering}.

\section*{Notations and conventions}
In this paper,~$k$ will always denote an algebraically closed field.  All modules will be right modules unless stated otherwise.  Arrows in a quiver are composed from right to left, as are morphisms in a category.  The suspension functor of a triangulated category will always be denoted by~$\Sigma$.  In any category~$\cC$, the set of morphisms from an object~$A$ to an object~$B$ will be denoted by~$\Hom{\cC}(A,B)$, or sometimes~$(A,B)$ to save space.  A general element of a (quasi-projective) variety will always mean an element of a dense open subset of the variety.

\section{Recollections on~$\bg$-vectors and varieties of representations}
\label{sect::recollections}

In this section, we fix a finite-dimensional basic~$k$-algebra~$\Lambda$.  

\subsection{Generic decomposition of~$\bg$-vectors}

Denote by~$\Lambda = \bigoplus_{i=1}^n P_i$ a decomposition of~$\Lambda$ as direct sum of pairwise non-isomorphic indecomposable projective right~$\Lambda$-modules.  
We will denote by~$K^b(\proj \Lambda)$ the homotopy category of bounded complexes of finitely generated projective right~$\Lambda$-modules.  We denote by~$K^{[-1,0]}(\proj \Lambda)$ the full subcategory whose objects are complexes concentrated in degrees~$-1$ and~$0$.  An object~$P$ in~$K^{[-1,0]}(\proj \Lambda)$ will be denoted by
\[
 P= P^{-1} \xrightarrow{f} P^0;
\]
by abuse of terminology, we will sometimes identify the object~$P$ with the morphism~$f$.

The Grothendieck group of the triangulated category~$K^b(\proj \Lambda)$ will be denoted by~$K_0(\proj\Lambda)$; we note that it is a free abelian group with basis the images of the indecomposable projectives~$P_1, \ldots, P_n$.  The image of an object~$X$ in the Grothendieck group will be denoted by~$[X]$.  The basis~$[P_1], \ldots, [P_n]$ gives a natural identification of~$K_0(\proj\Lambda)$ with~$\bZ^n$, which we will often use implicitly.

\begin{definition}
 A~\emph{$\bg$-vector} is an element of~$K_0(\proj\Lambda)$.  The~$\bg$-vector of an object~$P$ of~$K^{[-1,0]}(\proj \Lambda)$ is the element~$[P]\in K_0(\proj\Lambda)$.
\end{definition}

For any~$\bg\in K_0(\proj\Lambda)$, we let~$P^{\bg_+}$ and~$P^{\bg_-}$ be the unique finitely generated projective modules without common non-zero direct summands (up to isomorphism) such that~$\bg = [P^{\bg_+}] - [P^{\bg_-}]$.

\begin{definition}
 Let~$\bg,\bg'\in K_0(\proj\Lambda)$. We denote by~$e(\bg,\bg')$ the minimal value of
 \[
  \dim \Hom{K^b(\proj\Lambda)}(P,\Sigma P'),
 \]
 where~$P$ and~$P'$ are objects of~$K^{[-1,0]}(\proj\Lambda)$ such that~$[P]=\bg$ and~$[P']=\bg'$.
\end{definition}

The following result is due to H.~Derksen and J.~Fei~\cite[Theorem 4.4]{DerksenFei}.  We use the formulation given in~\cite[Theorem 2.7]{plamondon2013}.

\begin{theorem}[Generic decomposition of~$\bg$-vectors]\label{theo::generic2termcomplexes}
 Any~$\bg\in K_0(\proj\Lambda)$ can be written as
 \[
  \bg = \bg_1 + \ldots + \bg_r,
 \]
 where
 \begin{enumerate}
  \item for each~$i\in\{1, \ldots, r\}$, a general element of~$\Hom{\Lambda}(P^{(\bg_i)_-}, P^{(\bg_i)_+})$ is indecomposable;
  \item for each~$i,j\in \{1, \ldots, r\}$ with~$i\neq j$, we have that~$e(\bg_i, \bg_j) = 0$.
 \end{enumerate}
 Moreover, the elements~$\bg_1, \ldots, \bg_r$ are unique for these properties (up to re-ordering), and a general element of~$\Hom{\Lambda}(P^{\bg_-}, P^{\bg_+})$ is a direct sum of elements in~\[\Hom{\Lambda}(P^{(\bg_1)_-}, P^{(\bg_1)_+}), \ldots, \Hom{\Lambda}(P^{(\bg_r)_-}, P^{(\bg_r)_+}).\]
\end{theorem}

\begin{definition}
 Let~$\bg\in K_0(\proj\Lambda)$. The decomposition of~$\bg$ given in Theorem~\ref{theo::generic2termcomplexes} is the \emph{generic decomposition} of~$\bg$.  If~$r=1$, then~$\bg$ is \emph{generically indecomposable}.
\end{definition}

\subsection{Varieties of representations and generically~$\tau$-reduced components}
Since~$k$ is algebraically closed and~$\Lambda$ is basic, we can assume that~$\Lambda$ is the quotient of the path algebra of a finite quiver~$Q$ by an admissible ideal~$I$ (see, for instance, \cite{AssemSimsonSkowronski}).  Let~$e_1, \ldots, e_n$ be the paths of length~$0$ corresponding to the vertices of~$Q$; we can assume that~$P_i = e_i\Lambda$ for each~$i\in \{1, \ldots, n\}$.  Let~$S_1, \ldots, S_n$ be the simple tops of~$P_1, \ldots, P_n$, respectively.

We denote by~$\MOD \Lambda$ the category of finitely generated right~$\Lambda$-modules, and by~$K_0(\MOD \Lambda)$ its Grothendieck group.  Note that~$K_0(\MOD \Lambda)$ is a free abelian group, and that~$[S_1], \ldots, [S_n]$ form a basis for it.  Finally, let~
\[
 K_0(\MOD \Lambda)^{\oplus} := \big\{ \sum_{i=1}^n a_i [S_i] \ | \ a_1, \ldots, a_n\in \bZ_{\geq 0} \big\}.
\]

\begin{definition}
 A \emph{dimension vector} is an element of~$K_0(\MOD \Lambda)^{\oplus}$.  The \emph{dimension vector of a module}~$M$ is the element~$[M]\in K_0(\MOD \Lambda)^{\oplus}$. 
\end{definition}

For any~$\bd\in K_0(\MOD \Lambda)^{\oplus}$, we let~
\[
 \rep_{\bd}(Q^{\op}) := \bigoplus_{\alpha\in Q_1} \Hom{k}(k^{d_{t(\alpha)}}, k^{d_{s(\alpha)}}).
\]
It is an affine space whose points correspond to representations of the quiver~$Q^{\op}$ with dimension vector~$\bd$.  Inside it is the Zariski-closed subset $\rep_{\bd}(\Lambda)$ of representations satisfying the relations in the ideal~$I$; points in~$\rep_{\bd}(\Lambda)$ correspond to right~$\Lambda$-modules~$M$ such that~$[M] = \bd$.

The affine algebraic group~$\GL_{\bd} := \prod_{i\in Q_0} \GL_{d_i}(k)$ acts on~$\rep_{\bd}(\Lambda)$ by
\[
 (g_i)_{i\in Q_0} \cdot (f_{\alpha})_{\alpha\in Q_1} = (g_{s(\alpha)}f_{\alpha}g_{t(\alpha)}^{-1})_{\alpha\in Q_1}.
\]
The~$\GL_{\bd}$-orbits of~$\rep_{\bd}(\Lambda)$ are in bijection with isomorphism classes of right~$\Lambda$-modules with dimension vector~$\bd$.  The orbit of a point~$M$ will be denoted by~$\cO_M$.  By abuse of notation, we will identify a point~$M$ in~$\rep_{\bd}(\Lambda)$ with the module that it represents.

Let~$\cZ$ be an irreducible component of~$\rep_{\bd}(\Lambda)$.  It is known that~$\cZ$ is stable under the action of~$\GL_{\bd}$, and that for any point~$M\in \cZ$,
\[
 \codim_{\cZ} \cO_M \leq \dim_k \Ext{1}{\Lambda}(M,M) \leq \dim_k\Hom{\Lambda}(M, \tau M),
\]
where~$\tau$ is the Auslander--Reiten translation (here, the first inequality follows from Voigt's lemma \cite[Proposition 1.1]{Gabriel1974} and the second one from the Auslander--Reiten duality \cite[Theorem IV.2.13]{AssemSimsonSkowronski}).  The following definition is a dualized version of the notion of \emph{strongly reduced component} which was introduced by C.~Geiss, B.~Leclerc and J.~Schr\"oer in the context of additive categorification of cluster algebras and of G.~Lusztig's dual semicanonical bases.

\begin{definition}[Section 7.1 of \cite{GeissLeclercSchroer2012}]
 Let~$\bd$ be a dimension vector.  An irreducible component~$\cZ$ of~$\rep_{\bd}(\Lambda)$ is \emph{generically~$\tau$-reduced} if, for a general~$M\in \cZ$, we have that
 \[
  \codim_{\cZ} \cO_M = \dim_k\Hom{\Lambda}(M, \tau M).
 \]
\end{definition}

The next result states that generically~$\tau$-reduced components are parametrized by~$\bg$-vectors (in the case of~$\tau$-rigid modules, this result is due to~\cite{DehyKeller}; see also~\cite{DerksenFei} for generic decompositions of projective presentations).

\begin{theorem}[Theorem 1.2 of~\cite{plamondon2013}]\label{theo::genericCokernels}
 Let~$\bg\in K_0(\proj\Lambda)$.  There is an open dense subset~$\cU$ of~$\Hom{\Lambda}(P^{\bg_-}, P^{\bg_+})$ such that 
 \begin{enumerate}
  \item there is a dimension vector~$\bd\in K_0(\MOD\Lambda)^{\oplus}$ such that for any~$f\in \cU$, we have that~$[\Coker f] = \bd$;
  \item the union of the orbits of the cokernels of all~$f\in \cU$ is a dense subset of a generically~$\tau$-reduced component~$\cZ$ of~$\rep_{\bd}(\Lambda)$.
 \end{enumerate}
  Moreover, all generically~$\tau$-reduced components arise in this way, and two vectors~$\bg, \bg'\in K_0(\proj\Lambda)$ give rise to the same generically~$\tau$-reduced component if and only if their generic decompositions are the same, up to non-positive generically indecomposable terms.
\end{theorem}

\section{$\bg$-vector fans and tame algebras}
\label{sect::gVectorFansAndTameAlgebras}

\subsection{$\bg$-vector fans}

The~$\bg$-vectors of presilting objects of~$K^b(\proj\Lambda)$ are arranged into a structure called a simplicial fan.  We recall here the main definitions on fans, and refer to~\cite{ziegler1995} for a general treatment.

\begin{definition}
 Let~$d\in\bZ_{>0}$.  
 \begin{enumerate}
  \item For any non empty subset~$X$ of~$\bR^d$, a \emph{supporting hyperplane} of~$X$ is a hyperplane~$H$ of~$\bR^d$ such that~$H\cap X$ is non-empty and such that~$X$ is contained in one of the two half-spaces defined by~$H$.
  
  \item A \emph{polyhedral cone} in~$\bR^d$ is a set of the form
        \[
         C(c_1, \ldots, c_r) =  \big\{ \sum_{i=1}^r \lambda_i c_i \ | \ \lambda_i\in \bR_{\geq 0} \textrm{ for all $i$} \big\},   
        \]
        where~$c_1, \ldots, c_r$ are vectors in~$\bR^d$.          
        Equivalently, a polyhedral cone is the intersection of finitely many linear half-spaces.  
        The \emph{dimension} of~$C(c_1, \ldots, c_r)$ is the dimension of the vector space spanned by~$c_1, \ldots, c_r$.
        A cone~$C(c_1, \ldots, c_r)$ is \emph{simplicial} if~$c_1, \ldots, c_r$ are linearly independent.
        
  \item The \emph{faces} of a polyhedral cone~$C(c_1, \ldots, c_r)$ are its intersections with its supporting hyperplanes.  A \emph{ray} of a cone is a face of dimension~$1$; a \emph{facet} is a face of codimension~$1$.
  
  \item A \emph{polyhedral fan} in~$\bR^d$ is a set~$\cF$ of polyhedral cones in~$\bR^d$ such that
    \begin{itemize}
     \item if~$C\in \cF$, then any face of~$C$ is in~$\cF$; and
     \item the intersection of two cones in~$\cF$ is a face of both.
    \end{itemize}
  
  \item A polyhedral fan is \emph{simplicial} if all its cones are simplicial.  It is \emph{complete} if the union of its cones is~$\bR^d$.  It is \emph{essential} if it contains the cone~$\{0\}$.
 \end{enumerate}
\end{definition}

The following follows from the fundamental results of~$\tau$-tilting theory, see~\cite{AdachiIyamaReiten}.

\begin{theorem}[\cite{AdachiIyamaReiten}]\label{theo::silting-gvector-fan}
 Let~$\Lambda$ be any basic finite-dimensional~$k$-algebra.  Then the~$\bg$-vectors of indecomposable presilting objects of~$K^{[-1,0]}(\proj \Lambda)$ are the rays of an essential simplicial polyhedral fan whose maximal cones are~$C(g_1, \ldots, g_n)$, where~$g_1, \ldots, g_n$ are the~$\bg$-vectors of pairwise non-isomorphic indecomposable objects~$S_1, \ldots, S_n$ such that~$S = \bigoplus_{i=1}^n S_i$ is a silting object.
\end{theorem}

\begin{definition}
 The simplicial polyhedral fan described in Theorem~\ref{theo::silting-gvector-fan} is the \emph{$2$-term silting~$\bg$-vector fan} of~$\Lambda$, and it is denoted by~$\cF^{\bg}_{\tsilt}(\Lambda)$.
\end{definition}

\subsection{Tame algebras}
The main class of algebras that we will study is the following.
\begin{definition}
 The algebra~$\Lambda$ is \emph{tame} if for any dimension vector~$\bd$, there are~$k[t]$-$\Lambda$-bimodules~$M_1, \ldots, M_{m(\bd)}$ such that
 \begin{enumerate}
  \item for all~$i\in \{1, \ldots, m(\bd)\}$, the module~$M_i$ is free of finite rank as a~$k[t]$-module;
  \item all but finitely many indecomposable~$\Lambda$-modules of dimension vector~$\bd$ have the form
  \[
   k[t]/(t-\lambda) \otimes_{k[t]} M_i
  \]
  with~$i\in \{1, \ldots, m(\bd)\}$ and~$\lambda\in k$.
 \end{enumerate}
 For a given~$M_i$, the family of indecomposable~$\Lambda$-modules of the form~$k[t]/(t-\lambda) \otimes_{k[t]} M_i$ will be called a \emph{$1$-parameter family of indecomposable modules}.
 
\end{definition}

We will be using an important result on tame algebras, due to W.~Crawley-Boevey.

\begin{theorem}[Theorem D of~\cite{Crawley-Boevey}]
 Let~$\Lambda$ be a tame algebra.  Then for any dimension vector~$\bd$, all but finitely many isomorphism classes of~$\Lambda$-modules~$L$ of dimension vector~$\bd$ satisfy~$\tau L \cong L$.
\end{theorem}

\subsection{Generic decomposition of $\bg$-vectors for tame algebras}
In this section, we state a theorem due to C.~Geiss, D.~Labardini-Fragoso and J.~Schr\"oer\cite{GeissLabardiniFragosoSchroer2}.  We first recall some definitions.
\begin{definition}
 An object~$X$ of~$K^b(\proj\Lambda)$ is \emph{presilting} if~$\Hom{K^b(\proj\Lambda)}(X, X[i])$ vanishes for all~$i>0$.  It is called \emph{silting} if, moreover, it generates~$K^b(\proj\Lambda)$ as a triangulated category.
\end{definition}

\begin{definition}
 A~$\Lambda$-module is a \emph{brick} if its endomorphism algebra is isomorphic to the base field~$k$.
\end{definition}

\begin{theorem}[\cite{GeissLabardiniFragosoSchroer2}]\label{theo::genericDecompositionForTame}
 Let~$\Lambda$ be a tame algebra, and let~$\bg\in K_0(\proj\Lambda)$ be a~$\bg$-vector.  Then the generic decomposition of~$\bg$ has the form
 \[
  \bg = \bg_1 + \ldots + \bg_r + \bh_1 + \ldots + \bh_s,
 \]
 where $r,s\geq 0$ and
 \begin{enumerate}
  \item\label{item::g} for~$i\in \{1, \ldots, r\}$, the vector~$\bg_i$ is generically indecomposable and such that for a general~$f$ in~$\Hom{\Lambda}(P^{(\bg_i)_-}, P^{(\bg_i)_+})$, the object~$f$ of~$K^{[-1,0]}(\proj\Lambda)$ is presilting;
  \item\label{item::h} for~$j\in \{1, \ldots, s\}$, the vector~$\bh_j$ is generically indecomposable and there is a dense open subset~$\cU$ of~$\Hom{\Lambda}(P^{(\bh_j)_-}, P^{(\bh_j)_+})$ such that the cokernels of morphisms in~$\cU$ form a~$1$-parameter family of indecomposable~$\Lambda$-modules which are bricks and are isomorphic to their own Auslander-Reiten translate.   
  Moreover,~$e(\bh_j, \bh_j) =0$.
 \end{enumerate}
\end{theorem}

Theorem~\ref{theo::genericDecompositionForTame} was announced and its proof outlined in a lecture by J.~Schr\"oer at MFO in January 2020.  It is contained in~\cite{GeissLabardiniFragosoSchroer2}, where it is formulated in terms of generically~$\tau$-reduced components.  The above formultation in terms of generic decomposition of~$\bg$-vectors can be obtained by combining Theorem~3.2 and Section~4.2 of~\cite{GeissLabardiniFragosoSchroer2} together with~\cite[Theorem 4.4]{DerksenFei} and~\cite[Theorem 1.2]{plamondon2013}.
We include a complete proof here for the convenience of the reader.
\begin{proof}
 Assume first that~$\bg$ is generically indecomposable.  Let~$\cZ$ be the irreducible component of~$\rep_\bd(\Lambda)$ obtained by taking the closure of the union of the orbits of cokernels of generic elements in~$\Hom{\Lambda}(P^{\bg_-}, P^{\bg_+})$.  By Theorem~\ref{theo::genericCokernels}, the component~$\cZ$ is generically~$\tau$-reduced.
 
 Since~$\Lambda$ is tame, there exist~$k[t]$-$\Lambda$-bimodules~$M_1, \ldots, M_{m(\bd)}$, free of rank~$1$ as~$k[t]$-modules, such that almost all indecomposable~$\Lambda$-modules of dimension vector~$\bd$ are isomorphic to
 \[
  k[t]/(t-\lambda) \otimes_{k[t]} M_i
 \]
 for some~$\lambda\in k$ and some~$i\in\{1, \ldots, m(\bd)\}$.  Each~$M_i$ induces a morphism of varieties
 \[
  \phi_i : \bA^1 \to \rep_{\bd}(\Lambda)
 \]
 sending~$\lambda\in\bA^1$ to a point corresponding to the module~$k[t]/(t-\lambda) \otimes_{k[t]} M_i$.  Since~$\bA^1$ is irreducible, the image of each~$\phi_i$ either is contained in~$\cZ$ or has empty intersection with~$\cZ$.  Assume that the indices for which the image of~$\phi_i$ is contained in~$\cZ$ are~$1, \ldots, m'$.
 
 For each~$i\in\{1, \ldots, m'\}$, consider the morphism
 \begin{eqnarray*}
  \psi_i:\GL_\bd \times \bA^1 & \longrightarrow & \cZ \\
  (g,\lambda) & \longmapsto & g \cdot \phi_i(\lambda).
 \end{eqnarray*}
 By our assumptions on~$\cZ$, the union of the images of the~$\psi_i$ is dense in~$\cZ$.  Therefore, there is an~$\ell\in\{1, \ldots, m'\}$ such that the image of~$\psi_\ell$ is dense in~$\cZ$.

 Let~$L$ be a generic point in the image of~$\psi_\ell$; we shall make a small abuse of notation and denote by~$L$ the corresponding~$\Lambda$-module.  Assume that~$L = \psi_\ell(g,\lambda)$.  Then~$\psi_\ell(\GL_\bd \times \{\lambda\})$ is the orbit~$\cO_L$ of~$L$ in~$\cZ$.  Since~$\psi_\ell(\GL_\bd \times \bA^1)$ is dense in~$\cZ$, it is of codimension~$0$.  Thus,~$\cO_L$ is of codimension~$0$ or~$1$ in~$\cZ$.
 
 \bigskip
 
 {\it Case 1:~$\codim_\cZ \cO_L = 0$.}  Since~$\cZ$ is generically~$\tau$-reduced, we have the equalities~$\dim\Hom{\Lambda}(L,\tau L) = \codim_\cZ \cO_L = 0$.  Therefore,~$L$ is~$\tau$-rigid.  Since the projective presentations of~$\tau$-rigid modules are~$2$-term presilting objects, this shows that generic elements in the space~$\Hom{K^b(\proj\Lambda)}(P^{\bg_-}, P^{\bg_+})$ is presilting.  Thus~$\bg$ satisfies condition~(\ref{item::g}) of the Theorem.
 
 \bigskip
 
 {\it Case 2:~$\codim_\cZ\cO_L = 1$.}  In this case, since~$\cZ$ is generically~$\tau$-reduced, we get that~$\dim\Hom{\Lambda}(L,\tau L) = 1$.  In particular,~$\cZ$ contains infinitely many orbits.  By~\cite[Theorem D]{Crawley-Boevey}, since~$L$ is generic, we get that~$\tau L$ is isomorphic to~$L$.  Therefore~$\dim\Hom{\Lambda}(L,L) = 1$, and~$L$ is a brick.  Thus to show that~$\bg$ satisfies condition~(\ref{item::h}) of the Theorem, all that remains is to see that~$e(\bg,\bg)=0$.
 
 Let~$(f_1, f_2)$ be a generic element of~$\Hom{\Lambda}(P^{\bg_-},P^{\bg_+})\times \Hom{\Lambda}(P^{\bg_-},P^{\bg_+})$.  Then
 \begin{eqnarray*}
  e(\bg,\bg) &=& \dim \Hom{K^b(\proj\Lambda)}(f_1, \Sigma f_2) \\
             &=& \dim \Hom{\Lambda}\big(\Coker(f_2), \tau \Coker(f_1)\big) \\
             &=& \dim \Hom{\Lambda}\big(\Coker(f_2), \Coker(f_1)\big),
 \end{eqnarray*}
 where the second equality is by~\cite[Lemma 2.6]{plamondon2013} and the third one by~\cite[Theorem D]{Crawley-Boevey}.  Letting~$L_i = \Coker(f_i)$ for~$i\in\{1,2\}$, it suffices to show that~$\Hom{\Lambda}(L_2, L_1) = 0$.
 
 The following argument was suggested to us by A.~Skowro\'nski and G.~Zwara.  Since~$L_1$ and~$L_2$ are bricks (as proved above), we have that
 \[
  \rad \End{\Lambda}(L_1) = 0 = \rad \End{\Lambda}(L_2).
 \]
 Since the function~$(X,Y) \mapsto \dim \rad_\Lambda(X,Y)$ is upper semicontinuous on~$\cZ$, this implies that for generic~$L_1$ and~$L_2$, we have that~$\rad_\Lambda (L_2, L_1) = 0$.  Since~$L_1$ and~$L_2$ are indecomposable and non isomorphic, this in turn implies that~$\Hom{\Lambda}(L_2, L_1) = 0$.
 
 This proves that~$e(\bg, \bg) = 0$.
 
 \bigskip
 
 The theorem is now proved for a generically indecomposable~$\bg$.  If~$\bg$ is generically decomposable, simply apply the above to each term of its generic decomposition, and the theorem is proved.
\end{proof}

We will need a technical corollary of Theorem~\ref{theo::genericDecompositionForTame}.

\begin{corollary}\label{coro::factor-through-SigmaLambda}
 Let~$\Lambda$ be any algebra satisfying conditions~(\ref{item::g}) and~(\ref{item::h}) of Theorem~\ref{theo::genericDecompositionForTame}.  Then
 \begin{enumerate}
    \item for~$j\in\{1, \ldots, s\}$ and a general~\[h\in\Hom{\Lambda}(P^{(\bh_j)_-}, P^{(\bh_j)_+}),\] the space~$\Hom{K^b(\proj \Lambda)}(h,\Sigma h)$ is one-dimensional;
    \item for~$i\in\{1, \ldots, r\}$, $j\in\{1, \ldots, s\}$ and a general~
          \[
           (g,h) \in \Hom{\Lambda}(P^{(\bg_i)_-}, P^{(\bg_i)_+}) \times \Hom{\Lambda}(P^{(\bh_j)_-}, P^{(\bh_j)_+}),
          \]
          all morphisms in~$\Hom{K^b(\proj\Lambda)}(g,h)$ factor through an object of~$\add(\Sigma\Lambda)$;
     
    \item for~$j,\ell\in\{1, \ldots, s\}$, and a general
          \[
           (h,h') \in \Hom{\Lambda}(P^{(\bh_j)_-}, P^{(\bh_j)_+}) \times \Hom{\Lambda}(P^{(\bh_{\ell})_-}, P^{(\bh_{\ell})_+}),
          \]
     all morphisms in~$\Hom{K^b(\proj \Lambda)}(h,h')$ factor through an object of~$\add(\Sigma\Lambda)$. 
   \end{enumerate}
\end{corollary}
\begin{proof}
 To prove (1), we note that
 \begin{eqnarray*}
  \Hom{K^b(\proj \Lambda)}(h,\Sigma h) & \cong & \Hom{\Lambda}(\Ho^0h, \tau \Ho^0h) \quad\textrm{(by \cite[Lemma 2.6]{plamondon2013})} \\
  & \cong & \Hom{\Lambda}(\Ho^0h, \Ho^0 h) \quad \textrm{(since~$\Ho^0h \cong \tau\Ho^0 h$ by \cite{Crawley-Boevey})} \\
  & \cong & k \quad \textrm{(since~$\Ho^0h$ is a brick).}
 \end{eqnarray*}
 The proof of (2) is similar:
 \begin{eqnarray*}
  \Hom{K^b(\proj\Lambda)}(g,h)/(\Sigma \Lambda) & \cong & \Hom{\Lambda}(\Ho^0g, \Ho^0 h) \\
  & \cong & \Hom{\Lambda}(\Ho^0g, \tau\Ho^0 h) \quad \textrm{(since~$\Ho^0h \cong \tau\Ho^0 h$ by \cite{Crawley-Boevey})} \\
  & \cong & \Hom{K^b(\proj\Lambda)}(h, \Sigma g) \quad\textrm{(by \cite[Lemma 2.6]{plamondon2013})} \\
  & \cong & 0,
 \end{eqnarray*}
 where~$(\Sigma\Lambda)$ is the subspace of morphisms factoring through~$\add(\Sigma \Lambda)$.
 Finally, (3) is proved as follows:
 \begin{eqnarray*}
  \Hom{K^b(\proj\Lambda)}(h,h')/(\Sigma \Lambda) & \cong & \Hom{\Lambda}(\Ho^0h, \Ho^0 h') \\
  & \cong & \Hom{\Lambda}(\Ho^0h, \tau\Ho^0 h') \quad \textrm{(since~$\Ho^0h' \cong \tau\Ho^0 h'$ by \cite{Crawley-Boevey})} \\
  & \cong & \Hom{K^b(\proj\Lambda)}(h', \Sigma h) \quad\textrm{(by \cite[Lemma 2.6]{plamondon2013})} \\
  & \cong & 0.
 \end{eqnarray*}
\end{proof}

\subsection{$\bg$-tame algebras}

The property that we wish to study is encapsulated in the following definition.  Let $\Lambda$ be a finite-dimensional algebra over a field $k$.
We denote by $\bar \cF^{\bg}_{\tsilt}(\Lambda)$ the closure of the union of the cones of the $2$-term silting $\bg$-vector fan $\cF^{\bg}_{\tsilt}(\Lambda)$.

\begin{definition}[\cite{AokiYurikusa}]
The algebra $\Lambda$ is \emph{$\bg$-tame} if $\bar\cF^{\bg}_{\tsilt}(\Lambda)=\bR^n$.
\end{definition}

In Section~\ref{sect::densitiForMutationFinite}, we will need the following property of~$\bg$-tame algebras.

\begin{proposition}\label{prop::reduction by idemp}
Let $e$ be an idempotent of $\Lambda$.
If $\Lambda$ is $\bg$-tame, then so is $\Lambda/(e)$.
\end{proposition}

\begin{proof}
Let $|e|=n-m$.
Suppose that $\Lambda/(e)$ is not $\bg$-tame.
Then by \cite[Theorem 3.17]{Asai}, there is some $m$-dimensional cone $C$ in $K_0(\proj\Lambda/(e))\simeq\bR^m$ such that there is a non-zero $\theta$-semistable ($\Lambda/(e)$)-module $M(\theta)$ for any $\theta\in C\cap\bQ^m$. 
Let $\phi : K_0(\proj\Lambda)\simeq\bR^n\rightarrow\bR^m$ be a natural coordinate projection. 
Then $M(\theta)_{\Lambda}$ is a $\phi^{-1}(\theta)$-semistable $\Lambda$-module.
Thus there is a non-zero $\theta'$-semistable $\Lambda$-module for any $\theta'\in\phi^{-1}(C)\cap\bQ^n$, where $\phi^{-1}(C)$ is an $n$-dimensional cone in $K_0(\proj\Lambda)$.
By \cite[Theorem 3.17]{Asai} again, $\Lambda$ is not $\bg$-tame.
\end{proof}

\section{Density of the~$2$-term silting~$\bg$-vector fan}
\label{sect::densityOfFan}
Our main theorem is the following.

\begin{theorem}\label{theo::density}
 Let~$\Lambda$ be a tame basic finite-dimensional~$k$-algebra, and let~$n$ be the rank of its Grothendieck group.  Then any vector of~$\bQ^n$ is in the closure~$\bar \cF^{\bg}_{\tsilt}$ of the union of the cones of the~$2$-term silting~$\bg$-vector fan~$\cF^{\bg}_{\tsilt}$.  In particular,~$\Lambda$ is~$\bg$-tame.
\end{theorem}

This section is devoted to the proof of Theorem~\ref{theo::density}.

\subsection{Cylinders}

The proof of Theorem~\ref{theo::density} involves a variation on the notion of spherical twist.

\begin{definition}
Let~$\cT$ be an essentially small,~$\Hom{}$-finite, Krull-Schmidt, triangulated category with suspension functor~$\Sigma$.  Let~$X$ be an object of~$\cT$.  For any object~$U$ of~$\cT$, choose a basis~$(f_1, \ldots, f_d)$ of the space~$\Hom{\cT}(U,X)$ and a triangle
\[
 \Sigma^{-1} X^d \xrightarrow{} \cyl_{X}U \xrightarrow{} U \xrightarrow{f} X^d,
\]
where~$f = \begin{bmatrix} f_1 \\ f_2 \\ \vdots \\ f_d \end{bmatrix}$.                                 The object~$\cyl_X U$ is the \emph{cylinder of~$U$ with respect to~$X$}.
\end{definition}

\begin{remark}
\begin{enumerate}
\item The object~$\cyl_X U$ is only defined up to (non-unique) isomorphism.  Moreover, we do not define an action of~$\cyl_X$ on morphisms, so it is not a functor.
  \item The morphism~$U\xrightarrow{f} X^d$ is functorial in~$U$.  Indeed, the functor
\[
\Hom{\cT}(?,X):\cT\to (\MOD k)^{\op}
\]
admits the right
adjoint~$D(?)\otimes_k X \cong \Hom{k}(?,X)$, and the morphism~$U\xrightarrow{f} X^d$ is then the unit of 
adjunction
\[
Id_{\cT} \xrightarrow{} \Hom{k}\left( \Hom{\cT}(?,X), X\right)
\] 
applied to~$U$.

\item Our definition of~$\cyl_X$ is similar to, but different from, the (dual) definition of the twist~$\operatorname{Tw}_X$ (or spherical twist if~$X$ is a spherical object) \cite[Definition 2.7]{SeidelThomas}.  Indeed, the latter involves the morphisms in every degree from~$U$ to~$X$, while the definition of~$\cyl_X U$ only involves the morphisms in degree~$0$.
  \item In the Appendix~\ref{appendix}, Bernhard Keller constructs a truncated version~$t^0_X$ of the twist functor.  Applying Example~\ref{ex:trunc-twist}(2) 
 to the case where~$\cA = \Lambda$, we get that the truncated twist functor~$t^0_X$ and the operator~$\cyl_X$ act in the same way on certain objects; this happens, for instance, in the setting of Lemma~\ref{lemm::rigid-cylinders}.
 \end{enumerate}
\end{remark}

\subsection{Commuting cylinders}
The first key lemma in the proof of Theorem~\ref{theo::density} is a condition for cylinders to commute.
\begin{lemma}[Commuting cylinders]\label{lemm::commuting-cylinders}
 Let~$\cT$ be an essentially small,~$\Hom{}$-finite, Krull-Schmidt, triangulated category with suspension functor~$\Sigma$.  Let~$X$ and~$Y$ be non-isomorphic indecomposable objects of~$\cT$, and let~$U$ be any object of~$\cT$.  Assume that the following hold:
 \begin{enumerate}
  \item $\Hom{\cT}(X,\Sigma Y) = \Hom{\cT}(Y, \Sigma X) = 0$;
  \item for any morphism~$\phi\in \Hom{\cT}(U,X)$ and any morphism~$\psi\in \Hom{\cT}(X,Y)$, the composition~$\psi\phi$ vanishes;
  \item for any morphism~$\phi'\in \Hom{\cT}(U,Y)$ and any morphism~$\psi'\in \Hom{\cT}(Y,X)$, the composition~$\psi'\phi'$ vanishes.
 \end{enumerate}
 Then~$\cyl_X \cyl_Y U$ is isomorphic to~$\cyl_Y \cyl_X U$.
\end{lemma}
\begin{proof}
 Let
 \[
 \Sigma^{-1} X^d \xrightarrow{} \cyl_{X}U \xrightarrow{x} U \xrightarrow{f} X^d \quad \textrm{and} \quad
 \Sigma^{-1} Y^e \xrightarrow{} \cyl_{Y}U \xrightarrow{y} U \xrightarrow{g} Y^e
 \]
 be the triangles defining~$\cyl_X U$ and~$\cyl_Y U$.  Applying~$\Hom{\cT}(-, Y)$ to the first triangle, we get the exact sequence
 \[
  \Hom{\cT}(X^d,Y) \to \Hom{\cT}(U,Y) \xrightarrow{x^*} \Hom{\cT}(\cyl_X U, Y)\to \Hom{\cT}(\Sigma^{-1} X^d, Y).
 \]
  By (2), the leftmost morphism vanishes.  By (1), the space~$\Hom{\cT}(\Sigma^{-1} X^d, Y)$ vanishes.  Therefore~$\Hom{\cT}(U,Y) \xrightarrow{x^*} \Hom{\cT}(\cyl_X U, Y)$ is an isomorphism.
  
  By applying~$\Hom{\cT}(-,X)$ to the second triangle above, and using (3) instead of (2), we get that~$\Hom{\cT}(U,X) \xrightarrow{y^*} \Hom{\cT}(\cyl_Y U, X)$ is also an isomorphism.
  
  Since~$x^*$ is an isomorphism, we get that~$(g_1 x, \ldots, g_e x)$ is a basis of~$\Hom{\cT}(\cyl_X U, Y)$, and so the defining triangle for~$\cyl_Y \cyl_X U$ is isomorphic to
  \[
   \Sigma^{-1} Y^e \xrightarrow{} \cyl_Y\cyl_{X}U \xrightarrow{} \cyl_X U \xrightarrow{gx} Y^e.
  \]
  Similarly, since~$y^*$ is an isomorphism, we get that the defining triangle for~$\cyl_X\cyl_Y U$ is isomorphic to
  \[
\Sigma^{-1} X^d \xrightarrow{} \cyl_{X}\cyl_Y U \xrightarrow{} \cyl_Y U \xrightarrow{fy} X^d.
  \]
  Finally, applying the octahedral axiom to the composition~$gx$ yields an octahedron
  \begin{center}
  $\phantom{xxxx}$\begin{xy} 0;<1pt,0pt>:<0pt,-1pt>:: 
(105,0) *+{\Sigma\cyl_Y\cyl_X U} ="0",
(74,90) *+{\cyl_X U} ="1",
(207,90) *+{Y^e} ="2",
(0,69) *+{X^d} ="3",
(103,147) *+{U} ="4",
(133,69) *+{\Sigma\cyl_Y U} ="5",
"0", {\ar|+ "1"},
"2", {\ar "0"},
"3", {\ar^{} "0"},
"0", {\ar@{.>}^{} "5"},
"1", {\ar^{gx} "2"},
"3", {\ar|+ "1"},
"1", {\ar^x "4"},
"4", {\ar^g "2"},
"2", {\ar@{.>} "5"},
"4", {\ar^f "3"},
"5", {\ar@{.>}|+_{\Sigma (fy)}"3"},
"5", {\ar@{.>}^{-\Sigma y} |+"4"},
\end{xy} 
\end{center}
Thus we have a triangle
\[
 \Sigma^{-1} X^d \xrightarrow{} \cyl_Y\cyl_X U \xrightarrow{} \cyl_Y U \xrightarrow{fy} X^d.
\]
Comparing with the previous triangle (whose last morphism was also~$fy$), we get that the objects~$\cyl_{X}\cyl_Y U$ and~$\cyl_Y\cyl_X U$ are isomorphic.
\end{proof}

\subsection{Rigid cylinders}
The second key lemma in the proof of Theorem~\ref{theo::density} is a condition for cylinders of rigid objects to be rigid in the category~$K^{[-1,0]}(\proj\Lambda)$.

\begin{lemma}[Rigid cylinders]\label{lemm::rigid-cylinders}
 Let~$H$ be an indecomposable object of~$K^{[-1,0]}(\proj\Lambda)$ such that~$\Hom{K^b(\proj\Lambda)}(H, \Sigma H)$ is one-dimensional, and let~$U$ be another object of~$K^{[-1,0]}(\proj\Lambda)$ satisfying the following:
 \begin{enumerate}
  \item\label{item-rigid} $U$ is rigid, that is,~$\Hom{K^b(\proj\Lambda)}(U, \Sigma U) = 0$;
  \item\label{item-no-ext} $\Hom{K^b(\proj\Lambda)}(H, \Sigma U) = 0$;
  \item\label{item-injective} for any non-zero~$g\in \Hom{\cD\Lambda}(\Sigma H, \nu H)$ the induced morphism
   \[
    \Hom{K^b(\proj\Lambda)}(U, \Sigma H) \xrightarrow{g_*} \Hom{\cD\Lambda}(U, \nu H)
   \]
   is injective, where~$\nu=-\otimes_\Lambda^L D\Lambda$ is the Nakayama functor.  
 \end{enumerate}
  Then~$\cyl_{\Sigma H} U$ is in~$K^{[-1,0]}(\proj\Lambda)$ and also satisfies~(\ref{item-rigid}), (\ref{item-no-ext}) and (\ref{item-injective}).
\end{lemma}
\begin{proof}
 The triangle~$H^d \xrightarrow{} \cyl_{\Sigma H}U \xrightarrow{} U \xrightarrow{f} \Sigma H^d$ shows that the object~$\cyl_{\Sigma H}U$ is in~$K^{[-1,0]}(\proj\Lambda)$, since this category is closed under extensions.  Moreover, it induces the following commutative diagram with exact rows and columns (where we write~$(A,B)$ instead of~$\Hom{K^b(\proj\Lambda)}(A,B)$).
 \begin{center}
 \[  
  \!\!\!\!\!\!\!\!\!\!\!\!\!\!\!\!\!\!\!\!\!\!\!\!\!\!\!\!\!\!\!\!\!\!\!\begin{tikzcd}[scale=.5]
   &&&(H^d, U) \ar{r}\ar{d} & (\Sigma^{-1}U, U)\ar{d} \\
   (\Sigma H^d, \Sigma H^d) \ar{r}\ar{d} & (U, \Sigma H^d) \ar{r}\ar{d} & (\cyl_{\Sigma H}U, \Sigma H^d) \ar{r}\ar{d} & (H^d, \Sigma H^d) \ar{r}\ar{d} & (\Sigma^{-1}U, \Sigma H^d)\ar{d} \\
   (\Sigma H^d, \Sigma \cyl_{\Sigma H}U) \ar{r}\ar{d} & (U, \Sigma \cyl_{\Sigma H}U) \ar{r}\ar{d} & (\cyl_{\Sigma H}U, \Sigma \cyl_{\Sigma H}U) \ar{r}\ar{d} & (H^d, \Sigma \cyl_{\Sigma H}U) \ar{r}\ar{d} & (\Sigma^{-1}U, \Sigma \cyl_{\Sigma H}U)\ar{d} \\
   (\Sigma H^d, \Sigma U) \ar{r} & (U, \Sigma U) \ar{r} & (\cyl_{\Sigma H}U, \Sigma U) \ar{r} & (H^d, \Sigma U) \ar{r} & (\Sigma^{-1}U, \Sigma U) \\
  \end{tikzcd} 
 \]
 \end{center}
 
 All spaces in the rightmost column vanish.  Indeed, the top one vanishes by condition (\ref{item-rigid}), and all others vanish since for any pair objects~$A$ and~$B$ in~$K^{[-1,0]}(\proj\Lambda)$, the space~$(\Sigma^{-1}A, \Sigma B)$ vanishes.  Moreover, the spaces~$(U,\Sigma U)$ and~$(H^d, \Sigma U)$ of the bottom row vanish by conditions~(\ref{item-rigid}) and~(\ref{item-no-ext}).  Therefore, the space~$(\cyl_{\Sigma H}U, \Sigma U)$ vanishes as well.
 
 By construction, any morphism from~$U$ to~$\Sigma H$ factors through~$f$.  Thus the map~$(\Sigma H^d, \Sigma H^d) \to (U, \Sigma H^d)$ in the second row is surjective, and so the one immediately to its right vanishes.  This implies that~$(\cyl_{\Sigma H}U, \Sigma H^d) \to (H^d, \Sigma H^d)$ is an isomorphism.
 
 Using the above information, the commutative diagrams simplifies considerably.
 \begin{center}
 \[  
  \!\!\!\!\!\!\!\!\!\!\!\!\!\!\!\!\!\!\!\!\!\begin{tikzcd}[scale=.5]
   &&&(H^d, U) \ar{r}\ar{d} & 0 \\
   (\Sigma H^d, \Sigma H^d) \ar[->>]{r}\ar{d} & (U, \Sigma H^d) \ar[r,"0"]\ar{d} & (\cyl_{\Sigma H}U, \Sigma H^d) \ar[r,"\sim"]\ar{d} & (H^d, \Sigma H^d) \ar{r}\ar{d} & 0 \\
   (\Sigma H^d, \Sigma \cyl_{\Sigma H}U) \ar{r}\ar{d} & (U, \Sigma \cyl_{\Sigma H}U) \ar{r}\ar{d} & (\cyl_{\Sigma H}U, \Sigma \cyl_{\Sigma H}U) \ar{r}\ar{d} & (H^d, \Sigma \cyl_{\Sigma H}U) \ar{r}\ar{d} & 0 \\
   (\Sigma H^d, \Sigma U)  & 0  & 0  & 0  &  \\
  \end{tikzcd} 
 \]
 \end{center}

 We are now ready to prove that~$\cyl_{\Sigma H}U$ satisfies properties~(\ref{item-rigid}) to~(\ref{item-injective}).
 
 \smallskip

 \noindent\emph{Proof of~(\ref{item-no-ext}).}  To prove that~$(H,\Sigma \cyl_{\Sigma H}U)$ vanishes, it suffices to prove that the map~$(H^d,U)\xrightarrow{f_*} (H^d, \Sigma H^d)$ in the rightmost non-zero column is surjective.  Applying the duality~$D$, this is equivalent to showing that the map~$D(H^d, \Sigma H^d) \xrightarrow{Df_*} D(H^d,U)$ is injective.  Applying the properties of the Nakayama functor~$\nu$, this is equivalent to proving that the map~$(\Sigma H^d, \nu H^d) \xrightarrow{f^*} (U, \nu H^d)$ is injective.
 
 Now,~$(H, \Sigma H)$ is one dimensional, and therefore, so is~$(\Sigma H, \nu H)$.  Let~$g$ be a non-zero element in~$(\Sigma H, \nu H)$; it is unique up to rescaling.  An element~$\phi$ of~$(\Sigma H^d, \nu H^d)$ can be viewed as a~$d\times d$ matrix with multiples of~$g$ as entries, say~$\phi = (\phi_{i,j}g)_{i,j}$.  Then
 \[
  f^*(\phi) = \phi f = \begin{bmatrix}
                        \phi_{1,1}g & \cdots & \phi_{1,d}g \\
                        \vdots & \ddots & \vdots \\
                        \phi_{d,1}g & \cdots & \phi_{d,d}g
                       \end{bmatrix} \begin{bmatrix}
                                      f_1 \\ \vdots \\ f_d
                                     \end{bmatrix} = \begin{bmatrix}
                                                      \sum_{i=1}^d \phi_{1,i}gf_i \\
                                                       \vdots \\
                                                      \sum_{i=1}^d \phi_{d,i}gf_i
                                                     \end{bmatrix}
                                                     =   \begin{bmatrix}
                                                          g\big(\sum_{i=1}^d \phi_{1,i}f_i\big) \\
                                                          \vdots \\
                                                          g\big(\sum_{i=1}^d \phi_{d,i}f_i\big)
                                                         \end{bmatrix}.
 \]
 Thus~$f^*(\phi)$ vanishes if and only if~$g\big(\sum_{i=1}^d \phi_{j,i}f_i\big)$ vanishes for all~$j\in\{1, \ldots, d\}$.  By condition~(\ref{item-injective}), this only happens if~$\big(\sum_{i=1}^d \phi_{j,i}f_i\big)$ vanishes for all~$j$.  Since the~$f_i$ are linearly independent, this is only possible if all~$\phi_{i,j}$ vanish.  Thus~$f^*(\phi) = 0$ implies that~$\phi =0$, and~$f^*$ is injective.  By the above, this shows that the map~$(H^d,U)\xrightarrow{f_*} (H^d, \Sigma H^d)$ is surjective, and so~$(H,\Sigma \cyl_{\Sigma H}U)$ vanishes.
 
 \smallskip
 
 \noindent \emph{Proof of~(\ref{item-injective}).} As above, let~$g$ be a non-zero element of~$(\Sigma H, \nu H)$ (it is unique up to rescaling).  There is a bifunctorial non-degenerate bilinear form
 \[
  (-,-) : (A,B) \times (B, \nu A) \to k,
 \]
 where~$A,B$ can be any objects of~$K^b(\proj\Lambda)$ (see, for instance, \cite{Happel1988}).  In particular, if~$h$ is a non-zero element of~$(H, \Sigma H)$, then~$(h,g) \neq 0$, since~$(H,\Sigma H)$ is one-dimensional and the form is non-degenerate.  But the bifunctoriality of the bilinear form implies that~$(h,g) = (id_{H}, gh)$.  Thus~$gh\neq 0$.  This implies that the map~$(H,\Sigma H) \xrightarrow{g_*} (H, \nu H)$ is injective.
 
 We wish to prove that~$(\cyl_{\Sigma H}U,\Sigma H) \xrightarrow{g_*} (\cyl_{\Sigma H}U,\nu H)$ is injective.  Consider the following commutative diagram with exact rows:
 \begin{center}
 \begin{tikzcd}
  (\Sigma H^d, \Sigma H) \ar[r,"f^*",->>]\ar[d] & (U,\Sigma H) \ar[r,"0"]\ar[d] & (\cyl_{\Sigma H} U, \Sigma H) \ar[r,hookrightarrow]\ar[d,"g_*"] & (H^d, \Sigma H)\ar[d,"g_*",hookrightarrow] \\
  (\Sigma H^d, \nu H) \ar[r,"f^*"] & (U,\nu H) \ar[r] & (\cyl_{\Sigma H} U, \nu H) \ar[r] & (H^d, \nu H)
 \end{tikzcd}
 \end{center}
 The uppermost, leftmost morphism is surjective since all morphisms from~$U$ to~$\Sigma H$ factor through~$f$; thus the upper middle map is zero and the upper, rightmost map is injective.  The rightmost vertical morphism is injective by the above.  Thus the map~$(\cyl_{\Sigma H}U,\Sigma H) \xrightarrow{g_*} (\cyl_{\Sigma H}U,\nu H)$ is injective.
 
 \smallskip
 
 \noindent\emph{Proof of~(\ref{item-rigid}).}  As a consequence of~(\ref{item-no-ext}) proved above, we get a commutative square
 \begin{center}
  \begin{tikzcd}
   (U,\Sigma H^d) \ar[r,"0"]\ar[d,->>] & (\cyl_{\Sigma H} U, \Sigma H^d)\ar[d] \\
   (U, \Sigma \cyl_{\Sigma H} U) \ar[r,->>] & (\cyl_{\Sigma H} U, \Sigma \cyl_{\Sigma H} U)
  \end{tikzcd}
 \end{center}
 where the top map is zero and the left and bottom maps are surjective.  Thus the composition of the top and right morphism is both zero and surjective.  Thus the space~$(\cyl_{\Sigma H} U, \Sigma \cyl_{\Sigma H} U)$ vanishes.

\end{proof}

\subsection{$\bg$-vectors of cylinders}

\begin{lemma}\label{lemm::g-vector-iterated-cylinders}
 Let~$H$ be an indecomposable object of~$K^{[-1,0]}(\proj\Lambda)$ such that~$(H,\Sigma H)$ is one-dimensional, and let~$U$ be another object of~$K^{[-1,0]}(\proj\Lambda)$.  Then for any~$m\in\bZ_{>0}$, the object~$\cyl_{\Sigma H}^m U$ is in~$K^{[-1,0]}(\proj\Lambda)$, and if~$d = \dim (U, \Sigma H)$, then
 \[
  [\cyl_{\Sigma H}^m U] = [U] + md[H].
 \]
\end{lemma}
\begin{proof}
 The triangle~$H^d \to \cyl_{\Sigma H} U \to U \to \Sigma H^d$ implies that the object~$\cyl_{\Sigma H} U$ is in~$K^{[-1,0]}(\proj\Lambda)$, since this category is closed under extensions.  It also implies that~$[\cyl_{\Sigma H} U] = [U] + d[H]$.  Applying the functor~$(-,\Sigma H)$, we get an exact sequence
 \begin{center}
  \begin{tikzcd}
   (\Sigma H^d, \Sigma H) \ar[r,->>] & (U, \Sigma H)\ar[r,"0"] & (\cyl_{\Sigma H} U, \Sigma H)\ar[r] & (H^d, \Sigma H) \ar[r] & (\Sigma^{-1} U, \Sigma H).
  \end{tikzcd}
 \end{center}
 Since~$(\Sigma^{-1} U, \Sigma H) = 0$, we deduce that~$(\cyl_{\Sigma H} U, \Sigma H)$ and~$(H^d, \Sigma H)$ are isomorphic.  Thus~$\dim (\cyl_{\Sigma H} U, \Sigma H) = d$.  From there, we apply induction on~$m$ to get the desired equality.
\end{proof}

\begin{lemma}\label{lemm::g-vector-of-commuting-cylinders}
 Under the hypotheses of Lemma~\ref{lemm::commuting-cylinders}, if~$d_X = \dim (U,X)$ and~$d_Y = \dim (U,Y)$, then
 \[
  [\cyl_X \cyl_Y U] = [U] - d_X[X] - d_Y[Y] \in K_0(\cT).
 \]
\end{lemma}
\begin{proof}
 In the proof of Lemma~\ref{lemm::commuting-cylinders}, we have obtained triangle~$\Sigma^{-1}X^{d_X} \to \cyl_X U \to U \to X^{d_X}$ and~$\Sigma^{-1} Y^{d_Y} \xrightarrow{} \cyl_Y\cyl_{X}U \xrightarrow{} \cyl_X U \xrightarrow{gx} Y^{d_Y}$.  It follows from them that~$[\cyl_X U] = [U] - d_X[X]$ and~$[\cyl_Y\cyl_{X}U] = [\cyl_X U] - d_Y[Y]$.  Since the object~$\cyl_Y\cyl_{X}U$ is isomorphic to~$\cyl_X\cyl_{Y}U$ by Lemma~\ref{lemm::commuting-cylinders}, we get the desired result.
\end{proof}

\begin{lemma}\label{lemm::g-vectors-of-iterated-and-commuting-cylinders}
 Let~$\Lambda$ be any finite-dimensional algebra, and let~$H_1, \ldots, H_s$ be objects of~$K^{[-1,0]}(\proj\Lambda)$ such that
 \begin{itemize}
  \item for each~$i\in\{1, \ldots, s\}$, the object~$H=H_i$ satisfies the hypotheses of Lemma~\ref{lemm::g-vector-iterated-cylinders};
  \item for each pair of disctinct~$i,j\in \{1, \ldots, s\}$, the objects~$X=\Sigma H_i$ and~$Y=\Sigma H_j$ satisfy the hypotheses of Lemma~\ref{lemm::commuting-cylinders}.
 \end{itemize}
 Let~$U$ be an object of~$K^{[-1,0]}(\proj\Lambda)$, $d_i = \dim(U,\Sigma H_i)$,  and~$a_1, \ldots, a_s \in \bZ_{>0}$.  Then~$\cyl_{\Sigma H_s}^{a_s} \cdots \cyl_{\Sigma H_1}^{a_1} U$ is an object of~$K^{[-1,0]}(\proj\Lambda)$, and
 \[
  [\cyl_{\Sigma H_s}^{a_s} \cdots \cyl_{\Sigma H_1}^{a_1} U] = [U] + \sum_{i=1}^s a_i d_i [H_i].
 \]
\end{lemma}
\begin{proof}
 Apply Lemmas~\ref{lemm::g-vector-iterated-cylinders} and~\ref{lemm::g-vector-of-commuting-cylinders}, and induction on~$s$.
\end{proof}

\subsection{Proof of density}
We now turn to the proof of Theorem~\ref{theo::density}.  Assume that~$\Lambda$ is tame.  Let~$\bg\in \bQ^n$.  To prove that~$\bg$ is in the closure~$\bar \cF^{\bg}_{\tsilt}(\Lambda)$, it suffices to prove that it is true for a positive scalar multiple of~$\bg$.  Thus, up to positive rescaling, we can assume that~$\bg\in \bZ^n$.

We apply Theorem~\ref{theo::genericDecompositionForTame} to~$\bg$, and let 
\[
 \bg = a_1 \bg_1 + \ldots + a_r \bg_r + b_1 \bh_1 + \ldots + b_s \bh_s
\]
be the generic decomposition of~$\bg$.  Here, we have grouped the terms so that~$i\neq j$ implies~$\bg_i\neq \bg_j$ and~$\bh_i \neq \bh_j$, with all~$a_i$ and~$b_i$ in~$\bZ_{>0}$.

If~$s=0$ (that is, if no term of the form~$\bh_i$ appears in the generic decomposition of~$\bg$), then~$\bg\in \cF^{\bg}_{\tsilt}(\Lambda)$ and there is nothing to prove.  Assume then that~$s>0$.  Note that, by~\cite[Theorem 6.1]{CerulliLabardiniSchroer}, this implies that~$r<n$.

For any~$i\in\{1,\ldots,s\}$, let~$H_i$ be a generic object with~$\bg$-vector~$\bh_i$.

\begin{lemma}\label{lemm::commuting-for-H}
 The objects~$H_1, \ldots, H_s$ satisfy the hypotheses of Lemma~\ref{lemm::g-vectors-of-iterated-and-commuting-cylinders}.


\end{lemma}
\begin{proof}
 To prove that each~$H_i$ satisfies the hypotheses of Lemma~\ref{lemm::g-vector-iterated-cylinders}, one simply notes that~$\dim(H_i, \Sigma H_i)$ by Corollary~\ref{coro::factor-through-SigmaLambda}(1).
 
 Let us prove that the hypotheses of Lemma~\ref{lemm::commuting-cylinders} are satisfied with~$X=\Sigma H_i$ and~$Y=\Sigma H_j$.  Condition (1) is satisfied since~$\bh_i$ and~$\bh_j$ are two distinct terms in the generic decomposition of~$\bg$.  
 
 To prove condition (2), recall from Corollary~\ref{coro::factor-through-SigmaLambda}(3) that any morphism from~$H_i$ to~$H_j$ factors through~$\add(\Sigma \Lambda)$.  Thus any morphism from~$\Sigma H_i$ to~$\Sigma H_j$ factors through~$\add(\Sigma^2 \Lambda)$, which implies that precomposing it with a morphism from~$U$ to~$\Sigma H_i$ will give zero.  Thus (2) is satisfied.
 
 The proof of condition (3) of Lemma~\ref{lemm::commuting-cylinders} is similar.
\end{proof}

Let~$G$ be a presilting object whose~$\bg$-vector is~$a_1\bg_1 + \ldots + a_r \bg_r$.  Let~$G'$ be its \emph{Bongartz co-completion}, defined by the triangle
\[
 \Lambda \xrightarrow{} G'' \xrightarrow{} G' \xrightarrow{} \Sigma \Lambda
\]
where the left-most morphism is a left~$\add G$-approximation of~$\Lambda$.  By the dual of \cite[Definition-Proposition 4.9]{Jasso}, the object~$G'\oplus G$ is a silting object in the category~$K^{[-1,0]}(\proj\Lambda)$.

\begin{lemma}\label{lemm::G'andHi}
 Taking~$U=G'$ and~$H=H_i$ (for~$i\in\{1, \ldots, s\}$), conditions~(\ref{item-rigid}) to (\ref{item-injective}) of Lemma~\ref{lemm::rigid-cylinders} are satisfied.  Moreover, the space~$\Hom{K^b(\proj\Lambda)}(G', \Sigma H_i)$ does not vanish.
\end{lemma}
\begin{proof}
 That~(\ref{item-rigid}) is true follows from the fact that~$G'$ is presilting.  Applying the functor~$\Hom{K^b(\proj\Lambda)}(H_i,-)$ to the triangle defining~$G'$, we get an exact sequence
 \[
  \Hom{K^b(\proj\Lambda)}(H_i,\Sigma G'') \to \Hom{K^b(\proj\Lambda)}(H_i,\Sigma G') \to \Hom{K^b(\proj\Lambda)}(H_i,\Sigma^2 \Lambda).
 \]
 The left term vanishes since~$\Hom{K^b(\proj\Lambda)}(H_i, \Sigma G)$ does, and the right term vanishes since~$\Lambda$ and~$H_i$ are in~$K^{[-1,0]}(\proj\Lambda)$.  Thus~$\Hom{K^b(\proj\Lambda)}(H_i,\Sigma G')=0$, and condition~(\ref{item-no-ext}) is true.
 
 To prove condition~(\ref{item-injective}), let~$g$ be a non-zero morphism from~$\Sigma H_i$ to~$\nu H_i$.  Then the map~$(\Sigma\Lambda, \Sigma H_i)\xrightarrow{g_*} (\Sigma\Lambda, \nu H_i)$ is injective.  Indeed, note that taking the cohomology in degree~$-1$ yields isomorphisms
 \[
  \Hom{K^b(\proj\Lambda)}(\Sigma\Lambda, \Sigma H_i) \cong \Hom{\Lambda}(\Lambda, H^{0} H_i)
 \]
 and
 \[
  \Hom{K^b(\proj\Lambda)}(\Sigma\Lambda, \nu H_i) \cong \Hom{\Lambda}(\Lambda, H^{-1}\nu H_i).
 \]
 But by Theorem~\ref{theo::genericDecompositionForTame}(2),~$H^{0} H_i$ and~$H^{-1}\nu H_i\cong \tau H^0 H_i$ are isomorphic as~$\Lambda$-modules.  Moreover, the morphism~$g$ induces a non-zero morphism~$H^{-1}g: H^{0} H_i \to H^{-1}\nu H_i$, which must thus be an isomorphism since these two isomorphic modules are bricks.  
 
 Therefore, if~$f:\Sigma\Lambda \to \Sigma H_i$ is a morphism such that~$gf =0$, then
 $H^{-1}g \circ H^{-1}f =0$, and since~$H^{-1}g$ is an isomorphism, then~$H^{-1}f =0$, and so~$f=0$ by the above.

 \bigskip
 
 Consider now the commutative diagram with exact rows
 \begin{center}
  \begin{tikzcd}
   & (\Sigma \Lambda, \Sigma H_i) \ar[d,"g_*",hookrightarrow]\ar[r] & (G', \Sigma H_i) \ar[d,"g_*"]\ar[r] & (G'',\Sigma H_i) \\
   (\Sigma G'',\nu H_i) \ar[r] & (\Sigma \Lambda,\nu H_i) \ar[r] & (G',\nu H_i)  &
  \end{tikzcd}
 \end{center}
 Now,~$(G'',\Sigma H_i)$ vanishes since~$(G, \Sigma H_i)$ does, and~$(\Sigma G'', \nu H_i)$ is isomorphic to the space~$D(H_i, \Sigma G'')$, which also vanishes.  Thus the bottom middle map is injective and the top middle map is surjective.  The left vertical map is injective by the above.  This implies that the right vertical map is injective, so condition~(\ref{item-injective}) is proved.  This also implies that the top middle map has to be injective, and is thus an isomorphism.  Thus~$(G', \Sigma H_i)$ is isomorphic to~$(\Sigma \Lambda, \Sigma H_i)$, which is non-zero. 

\end{proof}

\begin{lemma}\label{lemm::compatibility}
 Let~$U$ be an object of~$K^{[-1,0]}(\proj\Lambda)$ and~$i\in \{1, \ldots, s\}$.  If~$(G, \Sigma U) = (U,\Sigma G) = 0$, then~$(G, \Sigma \cyl_{\Sigma H_i}U) = (\cyl_{\Sigma H_i}U,\Sigma G) = 0$.
\end{lemma}
\begin{proof}
 Simply apply~$(G,-)$ and~$(-,\Sigma G)$ to the triangle
 \[
  H_i^d \to \cyl_{\Sigma H_i}U \to U \to \Sigma H_i^d.
 \]

\end{proof}

We can now prove that the~$\bg\in \bZ^n$ chosen at the begining of this section is in~$\bar \cF^{\bg}_{\tsilt}(\Lambda)$, thus finishing the proof of Theorem~\ref{theo::density}.
For each~$i\in \{1, \ldots, s\}$, let~$d_i = \dim (G', \Sigma H_i)$.  By Lemma~\ref{lemm::G'andHi}, the~$d_i$ are non-zero.  Let~$d=\prod_{i=1}^s d_i$, and let~$e_i = \frac{d}{d_i}$ for each~$i$.

The, combining Lemmas~\ref{lemm::G'andHi}, \ref{lemm::rigid-cylinders}, \ref{lemm::commuting-for-H} and~\ref{lemm::commuting-cylinders},
we get that
\[
 \cyl_{\Sigma H_s}^{b_se_s} \cdots \cyl_{\Sigma H_1}^{b_1e_1} G'
\]
is a presilting object in~$K^{[-1,0]}(\proj\Lambda)$.  Moreover, 
\[
 G^{\oplus d} \oplus \cyl_{\Sigma H_s}^{b_se_s} \cdots \cyl_{\Sigma H_1}^{b_1e_1} G'
\]
is presilting by Lemma~\ref{lemm::compatibility}.  Now, using Lemma~\ref{lemm::g-vectors-of-iterated-and-commuting-cylinders}, we get that
\begin{eqnarray*}
 [G^{\oplus d} \oplus \cyl_{\Sigma H_s}^{b_se_s} \cdots \cyl_{\Sigma H_1}^{b_1e_1} G'] &=& d[G] + [G'] + \sum_{i=1}^s b_ie_id_i[H_i] \\
 &=& d([G] + \sum_{i=1}^s b_i[H_i]) + [G'] \\
 &=& d\bg + [G'].
\end{eqnarray*}
Similarly, for any~$m\in \bZ_{>0}$, we have that
\[
 G^{\oplus dm} \oplus \cyl_{\Sigma H_s}^{mb_se_s} \cdots \cyl_{\Sigma H_1}^{mb_1e_1} G'
\]
is a presilting object whose~$\bg$-vector is~$md\bg + [G']$.

The vector~$\bg$ is in the line at the limit of those generated by~$md\bg + [G']$ as~$m$ goes to infinity.  Since each~$md\bg + [G']$ is the~$\bg$-vector of a presilting objects, these vectors are in the fan~$\cF^{\bg}_{\tsilt}(\Lambda)$.  Thus~$\bg \in \bar \cF^{\bg}_{\tsilt}(\Lambda)$.  This finishes the proof of Theorem~\ref{theo::density}.



\section{Density of~$\bg$-vector fans for cluster algebras and Jacobian algebras}
\label{sect::densitiForMutationFinite}

In this section, we apply our results above to the~$\bg$-vector fans arising from the theory of cluster algebras.

\subsection{$\bg$-vectors for cluster algebras and Jacobian algebras}

We first recall the definition of~$\bg$-vectors for cluster algebras~\cite{FominZelevinsky2007}.

\begin{definition}[Proposition 6.6 of \cite{FominZelevinsky2007}]
 Let~$Q$ be a quiver without loops and~$2$-cycles, and let~$Q_0=\{1,\ldots, n\}$ be the set of its vertices.  Let~$Q^{\prin}$ be the quiver obtained by adding a vertex~$i'$ and an arrow~$i'\to i$ for every vertex~$i$ of~$Q$.  The~\emph{$\bg$-vectors} for the cluster algebras of type~$Q$ are obtained by the following mutation rule:
 \begin{enumerate}
  \item $\big(Q^{\prin},(\be_1,...,\be_n)\big)$ is a~\emph{$\bg$-vector seed}, where~$\be_i$ is the elementary vector of~$\bZ^n$ in coordinate~$i$;
  \item if~$\big( R, (\bg_1,...,\bg_n)\big)$ is a~$\bg$-vector seed, then for any vertex~$k\in \{1,\ldots, n\}$, the \emph{mutation}~$\mu_k\big( R, (\bg_1,...,\bg_n)\big) = \big( R', (\bg'_1,...,\bg'_n)\big)$ is also a~$\bg$-vector seed, where
  \[
   \bg'_k = \begin{cases}
              \bg_\ell & \textrm{if~$\ell \neq k$;} \\
              -\bg_k + \sum_{i=1}^n [b_{i,k}]_+ \bg_i - \sum_{j=1}^n [b_{j,k}]_+ \begin{pmatrix}                                                                   
                b_{1,j} \\ b_{2,j} \\ \vdots \\ b_{n,j}
              \end{pmatrix} & \textrm{if~$\ell=k$},
             \end{cases}
  \]
  where~$b_{i,j} = \#\{\textrm{arrows }i\to j\textrm{ in }R\} - \#\{\textrm{arrows }j\to i\textrm{ in }R\}$, and where~$[z]_+ = \max(z,0)$ for any real number~$z$.  
 \end{enumerate}
 The vectors~$\bg_i$ that appear in any~$\bg$-vector seed obtained by successive mutations of the initial seed~$\big(Q^{\prin},(\be_1,...,\be_n)\big)$ are the~\emph{$\bg$-vectors} for the cluster algebra of type~$Q$; ~$\bg$-vectors belonging to a common seed are \emph{compatible}. 
\end{definition}

The compatibility relation for~$\bg$-vectors allows one to organize them in a fan; this follows from the proof of conjectures in~\cite{FominZelevinsky2007}.

\begin{theorem}[Consequence of Theorem 1.7 of \cite{DerksenWeymanZelevinsky2}]\label{theo::cluster-gvector-fan}
Let $Q$ be a quiver without loops and $2$-cycles.
Then the $\bg$-vectors of cluster variables of $\cA(Q)$ are the rays of a simplicial polyhedral fan whose maximal cones are generated by sets of compatible~$\bg$-vectors.
\end{theorem}

\begin{definition}
The simplicial polyhedral fan described in Theorem \ref{theo::cluster-gvector-fan} is called the \emph{cluster $\bg$-vector fan} associated with $Q$, and it is denoted by $\cF^{\bg}_{\cluster}(Q)$.
We also denote by $\bar \cF^{\bg}_{\cluster}(Q)$ the closure of the union of the cones of $\cF^{\bg}_{\cluster}(Q)$.
\end{definition}

\begin{example}\label{ex::rank2}
Let $m \in \bZ_{\ge 1}$ and $K_m$ be an $m$-Kronecker quiver, that is,
\[
K_m:=[
\xymatrix{1  \ar@<1mm>@{}[r]|{\vdots}\ar@<2.3mm>[r]\ar@<-2.5mm>[r] & 2}],
\]
where there are $m$ arrows between vertices $1$ and $2$.
In particular, $K_1$ is of type $A_2$ and $K_2$ is a Kronecker quiver.
The cluster $\bg$-vector fan $\cF^{\bg}_{\cluster}(K_m)$ is well known and given as in Figure \ref{Fig::m-Kronecker}.
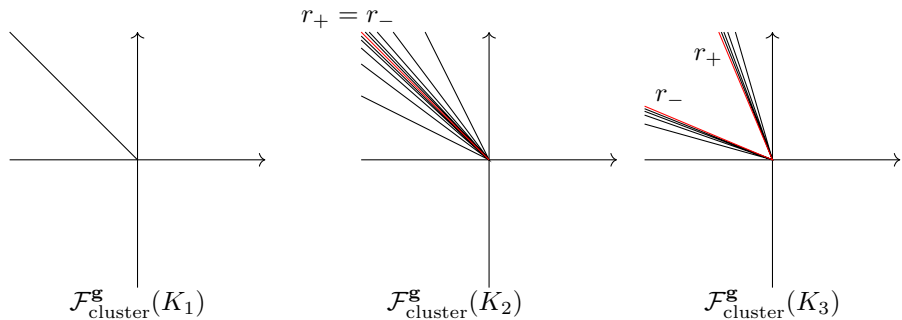
\begin{figure}[htp]
\centering
\begin{tabular}{ccc}
\begin{tikzpicture}[baseline=-1mm,scale=1.7]
 \coordinate (0) at (0,0); \coordinate (x) at (1,0); \coordinate (-x) at (-1,0);
 \coordinate (y) at (0,1); \coordinate (-y) at (0,-1);
 \draw[->] (0)--(x); \draw (0)--(-x); \draw[->] (0)--(y); \draw (0)--(-y);
 \draw(0)--(-1,1);
\end{tikzpicture}
&
\begin{tikzpicture}[baseline=-1mm,scale=1.7]
 \coordinate (0) at (0,0); \coordinate (x) at (1,0); \coordinate (-x) at (-1,0);
 \coordinate (y) at (0,1); \coordinate (-y) at (0,-1);
 \draw[->] (0)--(x); \draw (0)--(-x); \draw[->] (0)--(y); \draw (0)--(-y);
 \draw (-1,0.5)--(0)--(-0.5,1) (-1,0.75)--(0)--(-0.75,1) (-1,0.875)--(0)--(-0.875,1) (-1,0.9375)--(0)--(-0.9375,1) (-1,0.96875)--(0)--(-0.96875,1);
 \draw[red](0)--(-1,1); \node at(-1.1,1.1){$r_+=r_-$};
\end{tikzpicture}
&
\begin{tikzpicture}[baseline=-1mm,scale=1.7]
 \coordinate (0) at (0,0); \coordinate (x) at (1,0); \coordinate (-x) at (-1,0);
 \coordinate (y) at (0,1); \coordinate (-y) at (0,-1);
 \draw[->] (0)--(x); \draw (0)--(-x); \draw[->] (0)--(y); \draw (0)--(-y);
 \draw (-1,0.28)--(0)--(-0.29,1) (-1,0.35)--(0)--(-0.35,1) (-1,0.38)--(0)--(-0.38,1) (-1,0.4)--(0)--(-0.4,1); 
 \draw[red](-1,0.42)--(0)--(-0.42,1);
 \node at(-0.8,0.48){$r_-$}; \node at(-0.5,0.8){$r_+$};
\end{tikzpicture}\\
$\cF^{\bg}_{\cluster}(K_1)$ & $\cF^{\bg}_{\cluster}(K_2)$ & $\cF^{\bg}_{\cluster}(K_3)$
\end{tabular}
   \caption{Cluster $\bg$-vector fans $\cF^{\bg}_{\cluster}(K_m)$}
   \label{Fig::m-Kronecker}
\end{figure}
Here $\cF^{\bg}_{\cluster}(K_m)$ contains infinitely many rays converging to the rays $r_{\pm}$ of slope $(-m \pm \sqrt{m^2-4})/2$ for $m \ge 2$.
If $m=2$, then $r_+=r_-$.
If $m\ge 3$, then $r_+ \neq r_-$ and the interior of the cone spanned by $r_+$ and $r_-$ is the complement of $\bar \cF^{\bg}_{\cluster}(K_m)$.
\end{example}

Second, we recall Jacobian algebras \cite{DerksenWeymanZelevinsky}.
A \emph{potential} $W$ of $Q$ is a (possibly infinite) linear combination of cycles in $Q$. 
The pair $(Q,W)$ is called a \emph{quiver with potential}.
For a cycle $\alpha_1\cdots\alpha_m$ in $Q$ and an arrow $\alpha$ of $Q$, we define
\[
\partial_{\alpha}(\alpha_1\cdots\alpha_m):=\sum_{i:\alpha_i=\alpha}\alpha_{i+1}\cdots\alpha_m\alpha_1\cdots\alpha_{i-1}.
\]
By linearity, this defines the cyclic derivative $\partial_{\alpha}(W)$ of a potential $W$ of $Q$. 
For a quiver with potential $(Q,W)$, the \emph{Jacobian algebra} $J(Q,W)$ is a quotient algebra of the complete path algebra of $Q$ by the closure of the ideal generated by the set $\{\partial_{\alpha}W\mid\alpha\in Q_1\}$.
We say that a potential $W$ of $Q$ is \emph{Jacobi-finite} if $J(Q,W)$ is finite-dimensional.

The \emph{mutation $\mu_k(Q,W)$ in direction $k$} is defined as an analogue of mutations of quivers.
We refer to \cite{DerksenWeymanZelevinsky} for details.
A potential $W$ of $Q$ is called \emph{non-degenerate} if every quiver with potential obtained from $(Q,W)$ by any sequence of mutations has no $2$-cycles.
In this case, we have $\mu_k(Q,W)=(\mu_kQ,W')$ for a non-degenerate potential $W'$ of $\mu_kQ$.
If $W$ is Jacobi-finite, then so is $W'$ \cite{DerksenWeymanZelevinsky}.
In this section, we mainly study non-degenerate Jacobi-finite potentials.

On the other hand, mutations are defined for $2$-term silting complexes in \cite{IyamaYoshino}.
Let $\tsilt^+ J(Q,W)$ (resp., $\tsilt^- J(Q,W)$) be a subset of $\tsilt J(Q,W)$ consisting of objects obtained from $J(Q,W)$ (resp., $\Sigma J(Q,W)$) by all sequences of mutations.
The Jacobian algebra $J(Q,W)$ associated with a quiver with non-degenerate Jacobi-finite potential $(Q,W)$ gives a categorification of the associated cluster algebra $\cA(Q)$.
In particular, the following result is due to the authors \cite[Corollary 4.8]{AdachiIyamaReiten}, \cite[Theorem 6.3]{FuKeller} and \cite[Corollary 3.5]{IrelliKellerLabardiniPlamondon} (see also \cite[Theorem 4.4]{Yurikusa}).

\begin{theorem}[Additive categorification of cluster algebras]\label{theo::twosiltbijcluster}
Let $Q$ be a quiver without loops and $2$-cycles. 
Let $W$ be a non-degenerate Jacobi-finite potential of $Q$.
\begin{enumerate}
\item There is a bijection
\[
\tsilt^+ J(Q,W) \leftrightarrow \{\bg\textrm{-vector seeds of }Q\}
\]
that sends $J(Q,W)$ to the initial~$\bg$-vector seed for~$Q$ and commutes with mutations.
In particular, it preserves their cones of $\bg$-vectors.
\item There is a bijection
\[
x_{(-)} : \tsilt^- J(Q,W) \leftrightarrow \{\bg\textrm{-vector seeds of }Q^{\op}\}
\]
that sends $\Sigma J(Q,W)$ to the initial~$\bg$-vector seed for~$Q^{\op}$ and commutes with mutations.
In particular, the cone of~$\bg$-vectors associated to~$U\in \tsilt^- J(Q,W)$ is the negative of the cone of $\bg$-vectors associated to $x_U$.
\end{enumerate}
\end{theorem}

Theorem \ref{theo::twosiltbijcluster} immediately implies that the set
\[
\tsilt^{\pm} J(Q,W):=\tsilt^+ J(Q,W) \cup \tsilt^- J(Q,W)
\]
is independent of the choice of non-degenerate Jacobi-finite $W$.

\begin{corollary}\label{cor::differentpotentialpm}
Let $Q$ be a quiver without loops and $2$-cycles. 
Let $W$ and $W'$ be non-degenerate Jacobi-finite potentials of $Q$.
Then there is a bijection
\[
\tsilt^{\pm} J(Q,W) \leftrightarrow \tsilt^{\pm} J(Q,W')
\]
that sends $J(Q,W)$ to $J(Q,W')$ (resp., $\Sigma J(Q,W)$ to $\Sigma J(Q,W')$) and commutes with mutations.
In particular, it preserves their cones of $\bg$-vectors.
\end{corollary}

\subsection{Mutation-finite quivers}

In the rest of this section, we fix a quiver $Q$ without loops and $2$-cycles.
We say that $Q$ is
{
\setlength{\leftmargini}{5mm}
\begin{itemize}
\item[-] \emph{mutation equivalent to} $Q'$ if $Q$ is obtained from $Q'$ by a sequence of mutations;
\item[-] \emph{mutation-finite} if there are only finitely many quivers mutation equivalent to $Q$.
\end{itemize}}
Felikson, Shapiro and Tumarkin \cite{FeliksonShapiroTumarkin} classified mutation-finite quivers, see also \cite[Section 12]{fomin2008}.

\begin{theorem}\cite[Theorem 6.1]{FeliksonShapiroTumarkin}\label{theo::classification of mutation finite}
A mutation-finite quiver $Q$ is one of the followings:
\begin{itemize}
\item an $m$-Kronecker quiver $K_m$ with $m \ge 3$;
\item a quiver defined from a triangulated surface (see \cite{fomin2008});
\item a quiver mutation equivalent to one of the quivers $E_i$, $\tilde{E}_i$, $E_i^{(1,1)}$, $X_6$ and $X_7$ for $i\in\{6,7,8\}$ as in Figure \ref{fig::exceptional quivers}.
\end{itemize}
\end{theorem}

\begin{figure}[htp]
\begin{tabular}{c|c|c|c}
$E_6$ &
\begin{tikzpicture}[baseline=3mm,scale=0.7]
\node(0)at(0,0){$\bullet$}; \node(l1)at(-1,0){$\bullet$}; \node(l2)at(-2,0){$\bullet$}; \node(u1)at(0,1){$\bullet$};\node(r1)at(1,0){$\bullet$}; \node(r2)at(2,0){$\bullet$};
\draw[->](l2)--(l1); \draw[->](l1)--(0); \draw[->](u1)--(0); \draw[->](r2)--(r1); \draw[->](r1)--(0);
\end{tikzpicture}
& $E_6^{(1,1)}$ &
\begin{tikzpicture}[baseline=0mm,scale=0.7]
\node(l1)at(-0.5,0){$\bullet$}; \node(l2)at(-1.5,0){$\bullet$}; 
\node(u)at(0,1){$\bullet$}; \node(d)at(0,-1){$\bullet$};
\node(r1)at(0.5,0){$\bullet$}; \node(r2)at(1.5,0){$\bullet$}; 
\node(r3)at(2.5,0){$\bullet$}; \node(r4)at(3.5,0){$\bullet$};
\draw[->](l2)--(l1); \draw[->](u)--(l1); \draw[->](l1)--(d); \draw[->](u)--(r1); \draw[->](r1)--(d);
\draw[->](-0.06,-0.7)--(-0.06,0.7); \draw[->](0.06,-0.7)--(0.06,0.7);
\draw[->](r2)--(r1); \draw[->](r4)--(r3); \draw[->](u)--(r3); \draw[->](r3)--(d);
\end{tikzpicture}
\\\hline
$E_7$ &
\begin{tikzpicture}[baseline=3mm,scale=0.7]
\node(0)at(0,0){$\bullet$}; \node(l1)at(-1,0){$\bullet$}; \node(l2)at(-2,0){$\bullet$}; \node(u1)at(0,1){$\bullet$};
\node(r1)at(1,0){$\bullet$}; \node(r2)at(2,0){$\bullet$}; \node(r3)at(3,0){$\bullet$};
\draw[->](l2)--(l1); \draw[->](l1)--(0); \draw[->](u1)--(0); \draw[->](r2)--(r1); \draw[->](r1)--(0);
\draw[->](r3)--(r2);
\end{tikzpicture}
& $E_7^{(1,1)}$ &
\begin{tikzpicture}[baseline=0mm,scale=0.7]
\node(l1)at(-0.5,0){$\bullet$}; \node(l2)at(-1.5,0){$\bullet$}; \node(l3)at(-2.5,0){$\bullet$}; 
\node(u)at(0,1){$\bullet$}; \node(d)at(0,-1){$\bullet$};
\node(r1)at(0.5,0){$\bullet$}; \node(r2)at(1.5,0){$\bullet$}; 
\node(r3)at(2.5,0){$\bullet$}; \node(r4)at(3.5,0){$\bullet$};
\draw[->](l3)--(l2); \draw[->](l2)--(l1); \draw[->](u)--(l1); \draw[->](l1)--(d); \draw[->](u)--(r1); \draw[->](r1)--(d);
\draw[->](-0.06,-0.7)--(-0.06,0.7); \draw[->](0.06,-0.7)--(0.06,0.7);
\draw[->](r3)--(r2); \draw[->](r4)--(r3); \draw[->](u)--(r2); \draw[->](r2)--(d);
\end{tikzpicture}
\\\hline
$E_8$ &
\begin{tikzpicture}[baseline=3mm,scale=0.7]
\node(0)at(0,0){$\bullet$}; \node(l1)at(-1,0){$\bullet$}; \node(l2)at(-2,0){$\bullet$}; \node(u1)at(0,1){$\bullet$};
\node(r1)at(1,0){$\bullet$}; \node(r2)at(2,0){$\bullet$}; \node(r3)at(3,0){$\bullet$}; \node(r4)at(4,0){$\bullet$};
\draw[->](l2)--(l1); \draw[->](l1)--(0); \draw[->](u1)--(0); \draw[->](r2)--(r1); \draw[->](r1)--(0);
\draw[->](r3)--(r2); \draw[->](r4)--(r3);
\end{tikzpicture}
& $E_8^{(1,1)}$ &
\begin{tikzpicture}[baseline=0mm,scale=0.7]
\node(l1)at(-0.5,0){$\bullet$}; \node(l2)at(-1.5,0){$\bullet$};
\node(u)at(0,1){$\bullet$}; \node(d)at(0,-1){$\bullet$};
\node(r1)at(0.5,0){$\bullet$}; \node(r2)at(1.5,0){$\bullet$}; 
\node(r3)at(2.5,0){$\bullet$}; \node(r4)at(3.5,0){$\bullet$};
\node(r5)at(4.5,0){$\bullet$}; \node(r6)at(5.5,0){$\bullet$};
\draw[->](l2)--(l1); \draw[->](u)--(l1); \draw[->](l1)--(d); \draw[->](u)--(r1); \draw[->](r1)--(d);
\draw[->](-0.06,-0.7)--(-0.06,0.7); \draw[->](0.06,-0.7)--(0.06,0.7);
\draw[->](r3)--(r2); \draw[->](r4)--(r3); \draw[->](u)--(r2); \draw[->](r2)--(d);
\draw[->](r5)--(r4); \draw[->](r6)--(r5);
\end{tikzpicture}
\\\hline
$\widetilde{E}_6$ &
\begin{tikzpicture}[baseline=7mm,scale=0.7]
\node(0)at(0,0){$\bullet$}; \node(l1)at(-1,0){$\bullet$}; \node(l2)at(-2,0){$\bullet$}; \node(u1)at(0,1){$\bullet$};\node(r1)at(1,0){$\bullet$}; \node(r2)at(2,0){$\bullet$}; \node(u2)at(0,2){$\bullet$};
\draw[->](l2)--(l1); \draw[->](l1)--(0); \draw[->](u1)--(0); \draw[->](r2)--(r1); \draw[->](r1)--(0); 
\draw[->](u2)--(u1);
\end{tikzpicture}
& $X_6$ &
\begin{tikzpicture}[baseline=0mm,scale=0.8]
\node(0)at(0,0){$\bullet$}; \node(l)at(180:1){$\bullet$}; \node(lu)at(120:1){$\bullet$};
\node(r)at(0:1){$\bullet$}; \node(ru)at(60:1){$\bullet$}; \node(d)at(0,-1){$\bullet$};
\draw[->](0)--(l); \draw[->](lu)--(0); \draw[->](0)--(ru); \draw[->](r)--(0); \draw[->](d)--(0);
\draw[->](170:1)--(130:1); \draw[->](173:0.9)--(127:0.9);
\draw[->](50:1)--(10:1); \draw[->](53:0.9)--(7:0.9);
\end{tikzpicture}
\\\hline
$\widetilde{E}_7$ &
\begin{tikzpicture}[baseline=3mm,scale=0.7]
\node(0)at(0,0){$\bullet$}; \node(l1)at(-1,0){$\bullet$}; \node(l2)at(-2,0){$\bullet$}; \node(u1)at(0,1){$\bullet$};
\node(r1)at(1,0){$\bullet$}; \node(r2)at(2,0){$\bullet$}; \node(r3)at(3,0){$\bullet$}; \node(l3)at(-3,0){$\bullet$};
\draw[->](l2)--(l1); \draw[->](l1)--(0); \draw[->](u1)--(0); \draw[->](r2)--(r1); \draw[->](r1)--(0);
\draw[->](r3)--(r2); \draw[->](l3)--(l2);
\end{tikzpicture}
& $X_7$ &
\begin{tikzpicture}[baseline=0mm,scale=0.8]
\node(0)at(0,0){$\bullet$}; \node(l)at(180:1){$\bullet$}; \node(lu)at(120:1){$\bullet$};
\node(r)at(0:1){$\bullet$}; \node(ru)at(60:1){$\bullet$};
\node(dl)at(-120:1){$\bullet$}; \node(dr)at(-60:1){$\bullet$};
\draw[->](0)--(l); \draw[->](lu)--(0); \draw[->](0)--(ru); \draw[->](r)--(0); \draw[->](0)--(dr); \draw[->](dl)--(0);
\draw[->](170:1)--(130:1); \draw[->](173:0.9)--(127:0.9);
\draw[->](50:1)--(10:1); \draw[->](53:0.9)--(7:0.9);
\draw[->](-70:1)--(-110:1); \draw[->](-67:0.9)--(-113:0.9);
\end{tikzpicture}
\\\hline
$\widetilde{E}_8$ &
\begin{tikzpicture}[baseline=0mm,scale=0.6]
\node(0)at(0,0){$\bullet$}; \node(l1)at(-1,0){$\bullet$}; \node(l2)at(-2,0){$\bullet$}; \node(u1)at(0,1){$\bullet$};
\node(r1)at(1,0){$\bullet$}; \node(r2)at(2,0){$\bullet$}; \node(r3)at(3,0){$\bullet$}; \node(r4)at(4,0){$\bullet$};
\node(r5)at(5,0){$\bullet$};
\draw[->](l2)--(l1); \draw[->](l1)--(0); \draw[->](u1)--(0); \draw[->](r2)--(r1); \draw[->](r1)--(0);
\draw[->](r3)--(r2); \draw[->](r4)--(r3); \draw[->](r5)--(r4);
\end{tikzpicture}
\end{tabular}
   \caption{Exceptional quivers}
   \label{fig::exceptional quivers}
\end{figure}
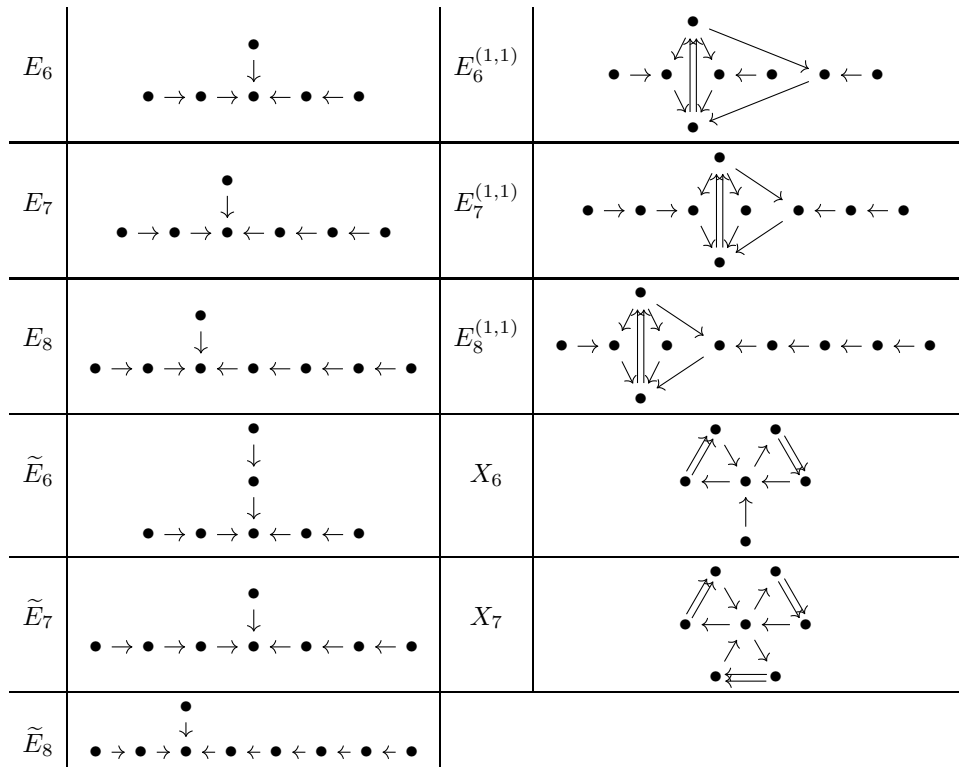

\begin{remark}
\begin{enumerate}
\item The quivers $K_m$, $E_i$ and $\tilde{E}_i$ are acyclic.
\item The Jacobian algebras $J$ and cluster categories associated with $E_i^{(1,1)}$, called \emph{tubular quivers}, were studied in \cite{BarotGeiss,BarotGeissJasso,GeissGonzalez}. 
In particular, it follows from \cite[Theorem 8.6]{BarotKussinLenzing} that any two silting objects of $K^{[-1,0]}(\proj J)$ are connected by a sequence of mutations, see \cite[Section 3]{BarotGeiss}.
\item Derksen and Owen \cite{DerksenOwen} found the quivers $X_6$ and $X_7$ as new mutation-finite quivers.
\end{enumerate}

\label{rem::mut-finquivers}
\end{remark}

Jacobian algebras associated with mutation-finite quivers satisfy some properties.

\begin{theorem}\cite{GeissLabardiniFragosoSchroer}\label{theo::jacobi-tame}
A quiver $Q$ is a mutation-finite one that is not mutation equivalent to one of the quivers $X_6$, $X_7$ and $K_m$ with $m \ge 3$ if and only if there is a non-degenerate Jacobi-finite potential $W$ of $Q$ such that $J(Q,W)$ is tame.
\end{theorem}

Note that, for example, such potential $W$ of acyclic quivers in Theorem \ref{theo::jacobi-tame} is zero.
In other cases, such potentials $W$ were given in \cite{Labardini1,Labardini2} for quivers defined from triangulated surfaces; in \cite{GeissGonzalez,Ladkani} for tubular quivers.

\begin{proposition}\label{prop::mutationtransitivity}
Suppose that $Q$ is mutation-finite except for mutation equivalence classes of $X_6$ and $X_7$.
Let $W$ be a non-degenerate Jacobi-finite potential of $Q$.
If $Q$ is not defined from a triangulated surface with exactly one puncture, then any two silting objects in $K^{[-1,0]}(\proj J(Q,W))$ are connected by a sequence of mutations, that is, $\tsilt J(Q,W)=\tsilt^+ J(Q,W)$. Otherwise, $\tsilt J(Q,W)=\tsilt^+ J(Q,W)\sqcup\tsilt^- J(Q,W)$.
\end{proposition}

\begin{proof}
By the assumption and Theorem \ref{theo::classification of mutation finite}, $Q$ is mutation equivalent to one of the following quivers: an acyclic quiver; a quiver defined from a triangulated surface; a tubular quiver.
For each case, the assertions follow from \cite[Proposition 3.5]{buan2006tilting}; \cite[Corollary 1.4]{Yurikusa}; Remark \ref{rem::mut-finquivers}(2).
\end{proof}

\begin{corollary}\label{cor::differentpotential}
Suppose that $Q$ is mutation-finite except for mutation equivalence classes of $X_6$ and $X_7$. 
Let $W$ and $W'$ be non-degenerate Jacobi-finite potentials of $Q$.
Then there is a bijection
\[
\tsilt J(Q,W) \leftrightarrow \tsilt J(Q,W')
\]
that sends $J(Q,W)$ to $J(Q,W')$ and commutes with mutations.
In particular, it preserves their cones of $\bg$-vectors.
\end{corollary}

\begin{proof}
The assertion follows from Corollary \ref{cor::differentpotentialpm} and Proposition \ref{prop::mutationtransitivity}.
\end{proof}

\subsection{$\bg$-tame Jacobian algebras}
In order to apply the categorification of cluster algebras to the study of the closure of their~$\bg$-vector fans, we need to study~$\bg$-vector fans for finite-dimensional Jacobian algebras.

\begin{lemma}\label{lem::g-tame mutation-inv}
If $(Q,W)$ is a quiver with potential such that~$J(Q,W)$ is~$\bg$-tame, then so is $\mu_k(Q,W)$ (if it is defined).
\end{lemma}

\begin{proof}
The transformation of mutations for $\bg$-vectors is piecewise linear.
Thus it preserves $\bg$-tameness.
\end{proof}

We note that the proof of Lemma~\ref{lem::g-tame mutation-inv} applies even if the Jacobian algebra~$J(Q,W)$ is infinite-dimensional.

We now give a near classification of quivers admitting a potential such that the corresponding Jacobian algebra is~$\bg$-tame.

\begin{theorem}\label{theo::classify g-tame J}
Suppose that $Q$ is not mutation equivalent to one of the quivers $X_6$, $X_7$ and $K_m$ with $m \ge 3$. 
Then $Q$ admits a non-degenerate potential~$W$ such that~$J(Q,W)$ is~$\bg$-tame if and only if it is mutation-finite.
Moreover, a Jacobian algebra for the quiver $K_m$ is not~$\bg$-tame for $m \ge 3$.
\end{theorem}

\begin{proof}
For $Q=K_m$, the assertion follows from Example \ref{ex::rank2}.
Suppose that $Q$ is not mutation equivalent to one of the quivers $X_6$, $X_7$ and $K_m$. 
When $Q$ is mutation-finite, then~$Q$ admits a potential~$W$ such that~$J(Q,W)$ is~$\bg$-tame by Theorems \ref{theo::density} and \ref{theo::jacobi-tame}.
When $Q$ is not mutation-finite, $Q$ is mutation equivalent to a quiver $Q'$ with full subquiver $K_m$ for $m \ge 3$.
Namely, there is an idempotent $e$ of $J(Q',W)$ for any non-degenerate potential $W$ of $Q'$ such that $J(Q',W)/(e) \simeq J(K_m,0) = kK_m$.
Proposition \ref{prop::reduction by idemp} implies that $J(Q',W)$ is not~$\bg$-tame.  Mutating back to~$Q$, we obtain the result by Lemma \ref{lem::g-tame mutation-inv}.
\end{proof}

\begin{corollary}
Suppose that $Q$ is not mutation equivalent to one of the quivers $X_6$ and $X_7$. 
If $J(Q,W)$ is~$\bg$-tame for some non-degenerate potential~$W$, then it is~$\bg$-tame for any non-degenerate Jacobi-finite potential.
\end{corollary}
\begin{proof}
If~$J(Q,W)$ is~$\bg$-tame, then by Theorem \ref{theo::classify g-tame J}, $Q$ is mutation-finite. The assertion then follows from Corollary \ref{cor::differentpotential}.
\end{proof}

\subsection{Cluster algebras with dense~$\bg$-vector fans}

We now wish to apply our knowledge of~$\bg$-vector fans for Jacobian algebras to those of cluster algebras.  To this end, the following definition will be useful.

\begin{definition}
We say that $Q$ is \emph{cluster-$\bg$-dense} if $\bar\cF^{\bg}_{\cluster}(Q)=\bR^n$.
We say that $Q$ is \emph{half cluster-$\bg$-dense} if $\bar\cF^{\bg}_{\cluster}(Q)$ and~$\bar\cF^{\bg}_{\cluster}(Q^{op})$ are closed half-spaces in $\bR^n$.
\end{definition}

In the same as Lemme \ref{lem::g-tame mutation-inv}, (half) cluster-$\bg$-denseness is mutation-invariant.
In particular, it gives a new class of cluster algebras, called \emph{(half) $\bg$-dense cluster algebras}.
It is clear that the associated finite-dimensional Jacobian algebras are $\bg$-tame.

\begin{lemma}\label{lemm::clustergDenseImpliesJacobiDense}
If $Q$ is (half) cluster-$\bg$-dense, then for any non-degenerate Jacobi-finite potential~$W$, the algebra~$J(Q,W)$ is~$\bg$-tame.
\end{lemma}

\begin{proof}
The assertion follows from Theorem \ref{theo::twosiltbijcluster}.
\end{proof}

We give a near classification of (half) cluster-$\bg$-dense quivers, the only open cases being the mutation equivalence classes of $X_6$ and $X_7$.
It is a stronger statement of Theorem \ref{theo::classify g-tame J}.

\begin{theorem}\label{theo::classify g-dense cluster}
Suppose that $Q$ is not mutation equivalent to one of the quivers $X_6$, $X_7$ and $K_m$ with $m \ge 3$.
Then $Q$ is cluster-$\bg$-dense or half cluster-$\bg$-dense if and only if it is mutation-finite.  In this case, $Q$ is half cluster-$\bg$-dense if and only if it is defined from a triangulated surface with exactly one puncture.
On the other hand, $K_m$ is not (half) cluster-$\bg$-dense for $m \ge 3$.
\end{theorem}

\begin{proof}
The assertions follow from Proposition \ref{prop::mutationtransitivity}, Theorems \ref{theo::twosiltbijcluster}(1) and \ref{theo::classify g-tame J} if $Q$ is mutation-finite except for one defined from a triangulated surface with exactly one puncture.
In which case, it was proved in \cite[Theorem 1.2]{Yurikusa}.

On the other hand, \cite[Theorem 33]{Muller} implies that a quiver with full subquiver $K_m$ for $m \ge 3$ is not (half) cluster-$\bg$-dense.
Thus it is given by the same way as the proof of Theorem \ref{theo::classify g-tame J} that non-mutation-finite quivers are not (half) cluster-$\bg$-dense.
\end{proof}

We conjecture that~$X_6$ and~$X_7$ should not be exceptions to Theorem~\ref{theo::classify g-dense cluster}.
It is known that:
\begin{enumerate}
\item\cite{Mills} If $Q$ is mutation equivalent to the quiver $X_6$, then there is a cluster whose cone of $\bg$-vectors is $\bR^n_{\le 0}$;
\item\cite{Seven} If $Q$ is mutation equivalent to the quiver $X_7$, then $\cF^{\bg}_{\cluster}(Q)$ is contained in some open half-space in $\bR^n$.
\end{enumerate}
Therefore, the following seems natural.

\begin{conjecture}\label{conj::typex}
\begin{enumerate}
\item The quiver $X_6$ is cluster-$\bg$-dense.
\item The quiver $X_7$ is half cluster-$\bg$-dense.
\end{enumerate}
\end{conjecture}

Remark that the Jacobian algebras associated with $X_6$ and $X_7$ are wild \cite{GeissLabardiniFragosoSchroer}.
Thus we cannot apply Theorem \ref{theo::density} for these classes.

The above allows us to give a partial converse to Lemma~\ref{lemm::clustergDenseImpliesJacobiDense}.

\begin{corollary}
Suppose that $Q$ is not mutation equivalent to one of the quivers $X_6$, $X_7$ and $K_m$ with $m \ge 3$. 
Then the following are equivalent:
\begin{enumerate}
\item $Q$ is mutation-finite;
\item $Q$ admits a non-degenerate potential~$W$ such that $J(Q,W)$ is~$\bg$-tame;
\item $Q$ is cluster-$\bg$-dense or half cluster-$\bg$-dense.
\end{enumerate}
Moreover, $K_m$ with $m \ge 3$ is mutation-finite, but not (half) cluster-$\bg$-dense and the unique potential~$0$ on~$K_m$ is such that~$J(K_m,0)$ is not~$\bg$-dense.
\end{corollary}

\begin{proof}
The assertions follow from Theorems \ref{theo::classify g-tame J} and \ref{theo::classify g-dense cluster}.
\end{proof}

\section{Scattering diagrams}
\label{sect::scattering}

Our work in this paper allows for a modest contribution to the question of the equivalence between the cluster and the stability scattering diagrams for (half) cluster-$\bg$-dense quivers; our result will follow by using an argument due to L.~Mou.  The statement of the result, however, requires some recollections on scattering diagrams which will occupy us for the beginning of the section.
We refer to \cite{bridgeland2017,GHKK,Mou} for scattering diagrams.

\subsection{Consistent scattering diagrams}

Let $N \simeq \bZ^{\oplus n}$ be a free abelian group.
Fix a basis $(e_1,\ldots,e_n)$ of $N$.
We use the notations $M=\Hom{\bZ}(N,\bZ)$, $M_{\bR} = M \otimes_{\bZ} \bR$,
\[
N^{\oplus} := \biggl\{\sum_{i=1}^n a_i e_i \mid a_i \in \bZ_{\ge 0}\biggr\} \quad\text{and}\quad N^+:=N^{\oplus}\setminus\{0\},
\]
and consider an $N^+$-graded Lie algebra
\[
\fg = \bigoplus_{d \in N^+}\fg_d.
\]

First, we define $\fg$-scattering diagrams for $\fg$ with finite support, that is,
\[
\#\Supp(\fg) < \infty, \quad\text{where}\quad \Supp(\fg) := \{d \in N^+ \mid \fg_d \neq 0\}.
\]
In this case, $\fg$ is nilpotent and there is a unipotent algebraic group $G$ such that $\exp : \fg \rightarrow G$ is a bijection.

For a cone $\sigma \subset M_{\bR}$, there is a Lie subalgebra
\[
\fg_{\sigma}:=\bigoplus_{d \in N^+ \cap \sigma^{\perp}} \fg_d \subset \fg,
\]
where $\sigma^{\perp} := \{d \in N \mid m(d)=0 \text{ for any $m \in \sigma$}\}$.

Let $P$ be a finite subset of $N^+$.
For a partition $P=P_+ \sqcup P_0 \sqcup P_-$ with $P_0 \neq \emptyset$, there is a cone
\[
\{m \in M_{\bR} \mid m(P_+)>0, m(P_0)=0, m(P_-)<0\} \subset M_{\bR}.
\]
The set of all such cones forms a polyhedral fan $\fS_P$ in $M_{\bR}$.
Cones with codimension one are called \emph{walls}.
We denote by $\cW(\fS)$ the set of walls of $\fS$.

\begin{definition}
Suppose that $\fg$ has a finite support $S=\Supp(\fg)$. A \emph{$\fg$-scattering diagram} is a pair $\fD=(\fS_S,\phi_{\fD})$ with function $\phi_{\fD} : \cW(\fS_S) \rightarrow G$ such that for $\sigma \in \cW(\fS_S)$, $\phi_{\fD}(\sigma) \in \exp(\fg_{\sigma})$.
\end{definition}

Second, we consider the consistency of a $\fg$-scattering diagram $\fD=(\fS_S,\phi_{\fD})$ for $\fg$ with finite support $S=\Supp(\fg)$.
A \emph{$\fD$-generic curve} is a smooth curve $\gamma : [0,1] \rightarrow M_{\bR}$ such that
\begin{enumerate}
\item the endpoints $\gamma(0)$ and $\gamma(1)$ lie in cones of $\fS_S$ with dimension $n$,
\item $\gamma$ does not intersect cones of $\fS_S$ with codimension at least two,
\item $\gamma$ and walls of $\fS_S$ intersect transversally.
\end{enumerate}
Then there are finitely many points $0 < t_1 < \cdots < t_{l} < 1$ and walls $\sigma_1,\ldots,\sigma_l \in \cW(\fS_S)$ such that $\gamma(t_i) \in \sigma_i$ and for $t \in [0,1]\setminus\{t_1,\ldots,t_l\}$, $\gamma(t)$ lie in cones of $\fS_S$ with dimension $n$.
We define the \emph{path-ordered product}
\[
\Phi_{\fD}(\gamma):=\phi_{\fD}(\sigma_l)^{\epsilon_l}\cdots\phi_{\fD}(\sigma_1)^{\epsilon_1} \in G,
\]
where $\epsilon_i\in\{1,-1\}$ is the negative of sign of the derivative of $\gamma(t)$ at $t=t_i$.
We say that two $\fg$-scattering diagrams $\fD_1$ and $\fD_2$ are \emph{equivalent} if any $\fD_1$-generic and $\fD_2$-generic curve $\gamma$ satisfies $\Phi_{\fD_1}(\gamma)=\Phi_{\fD_2}(\gamma)$.

\begin{definition}
We say that a $\fg$-scattering diagram $\fD$ is \emph{consistent} if any $\fD$-generic curves $\gamma_1$ and $\gamma_2$ with same endpoints satisfy $\Phi_{\fD}(\gamma_1)=\Phi_{\fD}(\gamma_2)$.
\end{definition}

Let $\fD$ be a consistent $\fg$-scattering diagram.
For cones $\sigma_1$ and $\sigma_2$ of $\fS_S$ with dimension $n$, we define
\[
\Phi_{\fD}(\sigma_1,\sigma_2):=\Phi_{\fD}(\gamma),
\]
where $\gamma$ is any $\fD$-generic curve with $\gamma(0) \in \sigma_1$ and $\gamma(1) \in \sigma_2$.
It does not depend on $\gamma$.

Finally, we define (consistent) $\fg$-scattering diagrams for a general $\fg$.
We define a map $\delta : N\rightarrow\bZ$ by $(d_i)\mapsto\sum_{i=1}^nd_i$.
Let $\fg^{>k}:=\bigoplus_{\delta(d)>k}\fg_d$ be a Lie subalgebra of $\fg$.
Then $\fg^{\le k}:=\fg/\fg^{>k}$ is a nilpotent Lie algebra with finite support and there is the corresponding unipotent algebraic group $G^{\le k}$.
There is a bijection as sets between 
\[
\hat{\fg}:=\mathop{\lim_{\longleftarrow}}_k\fg^{\le k} \quad\text{and}\quad \hat{G}:=\mathop{\lim_{\longleftarrow}}_k G^{\le k}.
\]

For $i>j$ and a $\fg^{\le i}$-scattering diagram $\fD_{\fg^{\le i}}=(\fS,\phi)$, we have a $\fg^{\le j}$-scattering diagram $\pi^{ij}_{\ast}(\fD_{\fg^{\le i}})=(\fS,\pi^{ij}\phi)$, where $\pi^{ij}$ is a natural homomorphism from $G^{\le i}$ to $G^{\le j}$.

\begin{definition}
A \emph{$\fg$-scattering diagram} is a sequence of $\fg^{\le k}$-scattering diagrams $(\fD^{\le k})_{k \ge 1}$ such that $\pi^{ij}_{\ast}(\fD^{\le i})$ is equivalent to $\fD^{\le j}$ for $i>j$.
We say that it is \emph{consistent} if so is $\fD^{\le k}$ for any $k \ge 1$.
\end{definition}

Remark that a consistent $\fg$-scattering diagram is considered as the pair $(\fS,\phi)$ consisting of the collection $\fS$ of certain cones, that is not rational polyhedral in general, and a function $\phi : \fS \rightarrow \hat{G}$ (see \cite[Remark 2.26]{Mou}).

In this section, our main subjects of study are cluster scattering diagrams and stability scattering diagrams.
To define them, the following result plays an important role.

\begin{theorem}\cite[Proposition 3.4]{bridgeland2017}\cite[Theorem 2.1.6]{KontsevichSoibelman}
There is a bijection between the set of equivalence classes of consistent $\fg$-scattering diagrams and $\hat{G}$ as sets.
\end{theorem}

We denote by $\fD_g=(\fD_g^{\le k})_{k \ge 1}$ the consistent $\fg$-scattering diagram corresponding to $g\in\hat{G}$.

\subsection{Cluster scattering diagrams}

Let $Q$ be a quiver without loops and $2$-cycles and $|Q_0|=n$.
Assume that $N$ has a skew-symmetric form $\{-,-\}:N \times N\rightarrow\bZ$ given by
\[
\{e_i,e_j\}:=\#\{\text{arrows from $j$ to $i$ in $Q$}\}-\#\{\text{arrows from $i$ to $j$ in $Q$}\}
\]
and $\fg$ is skew-symmetric, that is, if $\{d_1,d_2\}=0$ for $d_1, d_2\in N$, $[\fg_{d_1},\fg_{d_2}]=0$.

For $m \in M_{\bR}$, there is a decomposition
\[
\fg=\fg^m_+\oplus\fg^m_0\oplus\fg^m_-,
\]
where
\[
\fg^m_{\pm}:=\bigoplus_{d \in N^+: \pm m(d)>0}\fg_d \quad\text{and}\quad \fg^m_{0}:=\bigoplus_{d \in N^+: m(d)=0}\fg_d.
\]
We denote by $\hat{G}^m_{\bullet}$ the subgroup of $\hat{G}$ induced by $\fg^m_{\bullet}$ for $\bullet\in\{+,-,0\}$.
This gives a unique decomposition of $\hat{G}\ \reflectbox{$\in$}\ g=g^m_+ \cdot g^m_0 \cdot g^m_-$, where $g^m_{\bullet} \in \hat{G}^m_{\bullet}$.
Thus there is a projection map $\pi_m : G\rightarrow\hat{G}^m_0$ given by $g \mapsto g^m_0$.
Moreover, for a map $p^{\ast} : N \rightarrow M$ given by $d \mapsto \{d,-\}$, there is a decomposition
\[
\fg^{p^{\ast}(d)}_0=\fg_d^{||}\oplus\fg_d^{\perp},
\]
where $\fg_d^{||}:=\bigoplus_{k=1}^{\infty}\fg_{kd}$.
This naturally gives group homomorphisms $r_d : \hat{G}^{p^{\ast}(d)}_0 \rightarrow \hat{G}_d^{||}$ and $\psi_d := r_d\pi_{p^{\ast}(d)} : \hat{G} \rightarrow \hat{G}_d^{||}$.

\begin{proposition}\cite[Proposition 3.3.2]{KontsevichSoibelman}
The map 
\[
\psi:=(\psi_d)_{d \in N^+:primitive} : \hat{G} \rightarrow \prod_{d \in N^+:primitive}\hat{G}_d^{||}
\]
is a bijection as sets.
\end{proposition}

Set the skew-symmetric $N^+$-graded Lie algebra
\[
\fg=\fg_{\rm cl}:=\bigoplus_{d \in N^+}\bQ x^d \quad\text{with}\quad [x^{d_1},x^{d_2}]=\{d_1,d_2\}x^{d_1+d_2} \quad\text{for $d_1, d_2 \in N^+$}. 
\]

\begin{definition}
Let $g=(g_d) \in \prod_{d \in N^+:primitive}\hat{G}_d^{||}$ given by
\[
 g_d=\left\{
 \begin{array}{ll}
  \exp\Biggl(\displaystyle{\sum_{k=1}^{\infty}\frac{(-1)^{k-1}x^{ks_i}}{k^2}}\Biggr) &\text{if $d=e_i$},\\
  \rm{id} &\text{otherwise}.
 \end{array}\right.
\]
Then the corresponding consistent $\fg_{\rm cl}$-scattering diagram $\fD_Q:=\fD_{\psi^{-1}(g)}$ is called the \emph{cluster scattering diagram} associated with $Q$.
\end{definition}

\subsection{Stability scattering diagrams}

Let $J$ be a Jacobian algebra, and denote by  $K_0(\mod J)\simeq\bZ^n=N$ its Grothendieck group.
In this subsection, we recall the stability scattering diagram associated with $J$.
We refer to \cite{bridgeland2017} for the details.

Let $\fM(J)$ be the moduli stack associated with $\mod J$.
There is a decomposition 
\[
\fM(J)=\coprod_{d \in N^{\oplus}}\fM(J)_d,
\]
where $\fM(J)_d$ is the moduli stack of $J$-modules with dimension $d$.
Let $K({\rm St}/\bC)$ be the Grothendieck ring of stacks over $\bC$.
The Grothendieck group $K({\rm St}/\fM(J))$ of stacks over $\fM(J)$ is an $N^{\oplus}$-graded $K({\rm St}/\bC)$-algebra.
Since there is a unique ring homomorphism from $K({\rm St}/\bC)$ to $\bQ(t)$ sending the classes of smooth projective varieties to their Poincar\'e polynomials, we have an $N^{\oplus}$-graded $\bQ(t)$-algebra
\[
H(J):=K({\rm St}/\fM(J))\otimes_{K({\rm St}/\bC)}\bQ(t),
\]
called the \emph{motivic Hall algebra} associated with $J$.
For $d \in N^{\oplus}$, let $H(J)_d$ be the subspace of $H(J)$ generated by elements with form $[X \rightarrow \fM(J)]$ factoring through the inclusion from $\fM(J)_d$ to $\fM(J)$.
Then
\[
\fg_{\rm Hall}(J) := \bigoplus_{d \in N^+}H(J)_d
\]
is an $N^+$-graded Lie algebra with commutator bracket.
We have an element of the prounipotent algebraic group $\exp(\hat{\fg}_{\rm Hall}(J))$
\[
1_{\fM(J)}:=[\fM(J)\overset{\rm id}{\rightarrow}\fM(J)] \in 1+\hat{\fg}_{\rm Hall}(J)\simeq\exp(\hat{\fg}_{\rm Hall}(J)).
\]
The consistent $\fg_{\rm Hall}$-scattering diagram $\fD_{1_{\fM(J)}}$ is called the \emph{Hall algebra scattering diagram} associated with $J$.

Define a subalgebra
\[
\bC_{\rm reg}(t):=\bC[t,t^{-1}][(1+t^2+\cdots+t^{2k})^{-1} \mid k \ge 1] \subset \bC(t).
\]
We denote by $\fg_{\rm reg}(J) \subset \fg_{\rm Hall}(J)$ the $\bC_{\rm reg}(t)$-submodule generated by $[X \rightarrow \fM(J)]$ with algebraic variety $X$.

\begin{theorem}[\cite{joyce2008,bridgeland2017}]  
Any Hall algebra scattering diagram is a $\fg_{\rm reg}(J)$-scattering diagram. Moreover, there is an $N^+$-graded Lie algebra homomorphism
\[
\cI : \fg_{\rm reg}(J) \rightarrow \fg_{\rm cl}.
\]
\end{theorem}

\begin{definition}
The consistent $\fg_{\rm cl}$-scattering diagram $\fD_J:=\cI(\fD_{1_{\fM(J)}})$ is called the \emph{stability scattering diagram} associated with $J$.
\end{definition}

\subsection{Cluster/Stability scattering diagrams}

Let $Q$ be a quiver without loops and $2$-cycles and $W$ a non-degenerate potential of $Q$.
We state the main result of this section.

\begin{theorem}
\begin{enumerate}
\item If $Q$ is cluster-$\bg$-dense, then the cluster scattering diagram $\fD_Q$ is equivalent to the stability scattering diagram $\fD_{J(Q,W)}$.
\item If $Q$ is half cluster-$\bg$-dense, then $\fD_Q$ and $\fD_{J(Q,W)}$ only differ by functions on walls in the separating hyperplane.
\end{enumerate}
\label{theo::scatteringdiags}
\end{theorem}

To prove Theorem \ref{theo::scatteringdiags}, we need some preparations.
Set $J=J(Q,W)$ and $\fg=\fg_{\rm cl}$.
We only need to compare two consistent $\fg^{\le k}$-scattering diagrams $\fD_Q^{\le k}=(\fS_Q^{\le k},\phi_Q^{\le k})$ and $\fD_J^{\le k}=(\fS_J^{\le k},\phi_J^{\le k})$ for any $k \ge 1$.
Note that $\fg^{\le k}$ has finite support.

We can consider $\cF^{\bg}_{\cluster}(Q) \cup (-\cF^{\bg}_{\cluster}(Q^{\op}))$ as a fan in $M_{\bR} \simeq \bR^n$.
For a wall $\sigma \subset d_0^{\perp}$ of $\cF^{\bg}_{\cluster}(Q) \cup (-\cF^{\bg}_{\cluster}(Q^{\op}))$, where $d_0 \in N^+$ is primitive, let
\[
\phi(\sigma)=\exp\Biggl(\displaystyle{\sum_{k=1}^{\infty}\frac{(-1)^{k-1}x^{kd_0}}{k^2}}\Biggr) \in \hat{G}.
\]
For a projection map $p_k : \hat{G} \rightarrow G^{\le k}$, we define a fan $\cF^{\bg}_k(Q)$ consisting of cones $\sigma$ of $\cF^{\bg}_{\cluster}(Q) \cup (-\cF^{\bg}_{\cluster}(Q^{\op}))$ such that $p_k\phi(\sigma)$ is not trivial.
In particular, it has two cones
\[
M_{\bR}^{\pm}:=\{m \in M_{\bR} \mid \pm m(e_i) \ge 0\text{ for any $i$}\} \subset M_{\bR}.
\]

\begin{theorem}\cite[Theorem 4.25]{Mou}\label{theo::cluster subfan}
The fan $\cF^{\bg}_k(Q)$ is a common subfan of $\fS_Q^{\le k}$ and $\fS_J^{\le k}$ such that $\phi_Q^{\le k}(\sigma)=\phi_J^{\le k}(\sigma)=p_k\phi(\sigma)$ for any $\sigma \in \cW(\cF^{\bg}_k(Q))$.
\end{theorem}

We are ready to prove Theorem \ref{theo::scatteringdiags}.
For any wall $\sigma \in \cW(\fS_Q^{\le k})$, there are two cones $\sigma^+, \sigma^- \in G_Q^{\le k}$ with dimension $n$ such that $\sigma^+ \cap \sigma^- = \sigma$ since $\fS_Q^{\le k}$ is a finite fan.
If $Q$ is cluster-$\bg$-dense, then $\Phi_{G_Q^{\le k}}(M_{\bR}^+,\sigma^+)$ and $\Phi_{G_Q^{\le k}}(M_{\bR}^+,\sigma^-)$ are products of some functions on $\cF^{\bg}_k(Q)$.
Thus so is $\Phi_{G_Q^{\le k}}(\sigma^+,\sigma^-)$.
On the other hand, there is a $\fD_Q^{\le k}$-generic curve from $\sigma^+$ to $\sigma^-$ that intersects only one wall $\sigma$.
Thus we have $\Phi_{G_Q^{\le k}}(\sigma^+,\sigma^-) =\phi_{G_Q^{\le k}}(\sigma)^{\epsilon}$ for some $\epsilon \in \{1,-1\}$.
Therefore, for any wall $\sigma \in \cW(\fS_Q^{\le k})$, $\phi_{G_Q^{\le k}}(\sigma)$ is a product of some functions on $\cF^{\bg}_k(Q)$.

Similarly, we have that, for any wall $\sigma \in \cW(\fS_J^{\le k})$, $\phi_{G_J^{\le k}}(\sigma)$ is a product of some functions on $\cF^{\bg}_k(Q)$.
Consequently, Theorem \ref{theo::cluster subfan} implies Theorem \ref{theo::scatteringdiags}(1).

Suppose that $Q$ is half cluster-$\bg$-dense with the separating hyperplane $d^{\perp}$ for $d \in N^+$.
If $\sigma$ does not lie on $d^{\perp}$, then we have $\phi_{G_Q^{\le k}}(\sigma) = \phi_{G_J^{\le k}}(\sigma)$ in the same way as above.
Thus Theorem \ref{theo::scatteringdiags}(2) holds.

\section*{Acknowledgements}
We are grateful to A.~Skowro\'nski and G.~Zwara for suggesting an argument in the proof of Theorem~\ref{theo::genericDecompositionForTame}.  We also thank J.~Schr\"oer for discussing the results of~\cite{GeissLabardiniFragosoSchroer2} in a meeting at MFO in January 2020.  We thank B.~Keller for discussions on twists, for his comments on a previous version, and for agreeing to write the appendix to this paper.  Finally, we thank O.~Iyama for discussing the results of~\cite{AsaiDemonetIyama}, and K.~Mousavand for valuable discussions.  

\appendix

\section{Truncated twist functors, by Bernhard Keller}\label{appendix}

\subsection{Construction} We refer to \cite{Keller06d} for the terminology and
notation that we will use for dg (=differential graded) categories.
Let $k$ be a commutative ring and $\cS$, $\cT$
$k$-linear triangulated categories with dg enhancements $\cS_{dg}$ and
$\cT_{dg}$. Let $s: \cS \to \cT$ be a triangle functor induced by
a dg functor $s_{dg}: \cS_{dg} \to \cT_{dg}$. Assume that $s$ 
admits a right adjoint $r:\cT \to \cS$. Let there be given a $t$-structure
with truncation functor $\tau_{\leq 0}$ on $\cS$.

\begin{proposition} \label{prop:truncated-twist}
\begin{itemize} 
\item[a)] There is a triangle functor $t: \cT \to \cT$ admitting a dg enhancement
and fitting into a functorial triangle
\[
\xymatrix{ sr\ar[r] & \id_{\cT} \ar[r] & t \ar[r] & \Sigma sr}.
\]
\item[b)] There is a $k$-linear functor $t^0: \cT \to \cT$ fitting into a functorial
triangle
\[
\xymatrix{
s \tau_{\leq 0} r \ar[r] & \id_{\cT} \ar[r] & t^0 \ar[r] & \Sigma s\tau_{\leq 0} r.}
\]
\end{itemize}
\end{proposition}

The functor $t$ is called the {\em twist functor} and $t^0$ the {\em truncated
twist functor} associated with $s$.

\begin{examples}
\begin{enumerate}
\item Let $\cA$ be a dg category and $X$ an object of its derived
category $\cD\cA$. We can take $\cS=\cD k$, $\cT=\cD\cA$ and
$s=?\ltensor{k} X: \cD k \to \cD\cA$ with right adjoint
\[
r=\RHom{\cA}(X,?): \cD\cA \to \cD k.
\]
We get a twist functor $t_X$ and a truncated twist functor $t^0_X$ fitting
into functorial triangles
\[
\xymatrix{
\RHom{\cA}(X,?)\ltensor{k} X \ar[r] & \id_{\cD\cA} \ar[r] & t_X \ar[r] & }
\]
and
\[
\xymatrix{
\tau_{\leq 0}(\RHom{\cA}(X,?))\ltensor{k} X \ar[r] & \id_{\cD\cA} \ar[r] & t^0_X \ar[r] & }.
\]
\item With $\cA$, $X$ and $\cS$ as in example (1), we can take $\cT=(\cD\cA)^{op}$.
We have an adjoint pair
\[
\xymatrix{
(\cD\cA)^{op} \ar@<-1ex>[d]_{s=\RHom{\cA}(?,X)}\\
\cD k. \ar@<-1ex>[u]_{\RHom{k}(?,X)=r}
}
\]
This leads to the following functorial triangles in $\cD\cA$:
\[
\xymatrix{
\mbox{ } \ar[r] & t_X \ar[r] & \id_{\cD\cA} \ar[r] & \RHom{k}(\RHom{\cA}(?,X),X) 
}
\]
and 
\[
\xymatrix{
\mbox{ } \ar[r] & t_X^0 \ar[r] & \id_{\cD\cA} \ar[r] & \RHom{k}(\tau_{\leq 0}\RHom{\cA}(?,X), X).
}
\]
\end{enumerate}
\label{ex:trunc-twist}
\end{examples}

\begin{proof}[Proof of the Proposition] a) This is well-known, cf. for
example \cite{AnnoLogvinenko19}: One shows that one can replace
$\cT_{dg}$ and $\cS_{dg}$ with quasi-equivalent dg categories
such that $s_{dg}$ admits a dg adjoint functor $r_{dg}$.

b) Let us fix $s_{dg}$ and $r_{dg}$ as in a). From now on, we suppress
the subscript `dg' on these functors. Let $\cM_{dg}$ be the dg category
whose objects are the triples
\[
(U,V, f: sV \to U)
\]
where $U$ is in $\cT$, $V$ in $\cS$ and $f: sV \to U$ is a closed
morphism of degree $0$ in $\cS_{dg}$. By definition, the morphism
complex between two objects $(U,V,f)$ and $(U',V',f')$ is the cylinder
over the morphism
\[
[f^*, -f'\circ s(?)] : \cT_{dg}(U,U')\oplus \cS_{dg}(V,V') \to \cT_{dg}(sV,U').
\]
Thus, closed morphisms of degree $0$ are triples $(u,v,h)$ such that
$u:U \to U'$ is closed in $\cT_{dg}$, $v: V \to V'$ is closed in $\cS_{dg}$
and we have
\[
u\circ f - f'\circ s(v) = d(h)
\]
so that we have a homotopy commutative square
\[
\xymatrix{
sV \ar[d]_{s(v)} \ar[r]^f \ar@{-->}[dr]^h & U \ar[d]^u \\
sV' \ar[r]_{f'} & U'.
}
\]
Composition is defined in the natural way.
The projection functor $P: \cM_{dg} \to \cT_{dg}$ takes a triple
$(U,V,f)$ to $U$. It is easy to see that it has a fully faithful right dg adjoint $P_\rho$
taking $U$ to $(U,rU, srU \to U)$, where the last arrow is the adjunction
morphism, and a fully faithful left dg adjoint $P_\lambda$ taking
$U$ to $(U,0,0 \to U)$. The functor $I$ taking $V$ to $(0,V, sV \to 0)$
identifies $\cS_{dg}$ with the kernel of $P$.
We can construct the dg twist functor $t$ as the composition
\[
\xymatrix{ \cT_{dg} \ar[r]^{P_\rho} & \cM_{dg} \ar[r]^{\mbox{\tiny cone}} & \cT_{dg}.}
\]
Now let $\cM=H^0(\cM_{dg})$. We have a recollement
\[
\begin{tikzcd}
  \cS\arrow{rr}[description]{I} &&\cM
  \arrow[yshift=-1.5ex]{ll}{I_\rho}
  \arrow[yshift=1.5ex]{ll}[swap]{I_\lambda}
  \arrow{rr}[description]{P} &&\cT\,.
  \arrow[yshift=-1.5ex]{ll}{P_\rho}
  \arrow[yshift=1.5ex]{ll}[swap]{P_\lambda}
\end{tikzcd}
\]
Let $\cS_{\leq 0}$ be the left aisle of the
given $t$-structure on $\cS$. Using Th\'eor\`eme~1.4.10 of 
\cite{BeilinsonBernsteinDeligne82} or directly, we see that $\cM$ admits
a $t$-structure whose left aisle $\cM_{\leq 0}$ has the objects
$(U,V,f)$ such that $V\in \cS_{\leq 0}$. The corresponding morphism
$\tau_{\leq 0} \to \id_\cM$ is given on an object $(U,V,f)$ by
$(U,\tau_{\leq 0}V,g)$, where $g$ is the composition
\[
\xymatrix{s\tau_{\leq 0} V \ar[r]^{sc} & sV \ar[r]^f & U}
\]
for any closed morphism $c$ lifting the adjunction morphism 
$\tau_{\leq 0} V \to V$. We obtain the truncated twist functor $t^0$ as
the composition
\[
\xymatrix{
\cT \ar[r]^{P_\rho} & \cM \ar[r]^{\tau_{\leq 0}} & 
\cM_{\leq 0}  \ar[r]^{\mbox{\tiny cone}} & \cT.}
\]
\end{proof}

\subsection{Adjoints} We keep the assumptions and notations of the preceding
section. Suppose that $s: \cS \to \cT$ admits a left adjoint $l: \cT \to \cS$
and that the inclusion $\cS_{\leq 0} \to \cS$ also admits a left adjoint
$\tau_{\leq 0}'$ (for example, this holds if $k$ is a field and $\cS=\cD k$ is
endowed with the canonical $t$-structure).

\begin{proposition}
\begin{itemize}
\item[a)] The twist functor $t$ admits a left adjoint $t'$ fitting
into a functorial triangle
\[
\Sigma^{-1} sl \to t' \to \id_{\cT} \to sl.
\]
\item[b)] The truncated twist functor $t^0$ admits a left adjoint
$t'^0$ fitting into a functorial triangle
\[
\Sigma^{-1} s \tau'_{\leq -1} l \to t'^0 \to \id_\cT \to s\tau'_{\leq -1} l.
\]
\end{itemize}
\end{proposition}

\begin{examples} We keep the notations of Examples~\ref{ex:trunc-twist}.
We suppose that $k$ is a field.
\begin{enumerate} 
\item Suppose that $X\in \cD\cA$ is an object such that the homology of
$X(A)$ is of finite total dimension for all $A\in\cA$. Put
$D=\Hom{k}(?,k)$. Then for $U\in\cD\cA$
and $V\in \cD k$, we have canonical isomorphisms
\[
\RHom{\cA}(U, V\ten X) = \RHom{\cA}(U,\Hom{k}(DX,V))
= \Hom{k}(U\ltensor{\cA} DX, X).
\]
We get functorial triangles
\[
\xymatrix{
\mbox{ } \ar[r] & t' \ar[r] & \id_{\cD \cA} \ar[r] & (?\ltensor{\cA} DX)\ten X
}
\]
and
\[
\xymatrix{
\mbox{ } \ar[r] & t'^0 \ar[r] & \id_{\cD\cA} \ar[r] & (\tau'_{\leq -1}(?\ltensor{\cA} DX)) \ten X.}
\]

\item Let $\cT\subseteq \cD\cA$ be a Hom-finite triangulated subcategory
and $X$ an object in $\cT$ such that the homology of $X(A)$ is of finite
total diemsnion for each $A\in\cA$. Let
\[
s=\Hom{k}(?,X) : \cD^b(k) \to \cT^{op}.
\]
Then for $U$ in $\cT$ and $V$ in $\cD^b(k)$, we have
\begin{align*}
(\cD\cA)^{op}(U, \Hom{k}(V,X)) &= (\cD\cA)(\Hom{k}(V,X),U) \\
&= (\cD\cA)(X \ten DV, U) \\
&= (\cD\cA)(DV\ten X, U) \\
&= (\cD k)(DV,\RHom{\cA}(X,U)) \\
&= (\cD k)(D\RHom{\cA}(X,U), V).
\end{align*}
Thus, the functor $s: \cD^b(k) \to \cT$ admits $l=D\RHom{\cA}(X,?)$ as
a left adjoint and we get functorial triangles
\[
\xymatrix{
\mbox{ } \ar[r] & t' \ar[r] & \id_{\cT} \ar[r] & D\RHom{\cA}(?,X) \ten X}
\]
and
\[
\xymatrix{
\mbox{ } \ar[r] & t'^0 \ar[r] & \id_{\cT} \ar[r] & \tau'_{\leq -1}(D\RHom{\cA}(?,X)) \ten X.}
\]
\end{enumerate}
\end{examples}

\begin{proof}[Proof of the Proposition] Part a) is well-known, cf. for
example \cite{AnnoLogvinenko19}. Let us prove part b).
For objects $U,U'$ of $\cT$, we abbreviate
\[
(U,U') = \RHom{\cT_{dg}}(U,U')
\]
and similarly
\[
(V,V')=\RHom{\cS_{dg}}(V,V')
\]
for objects $V,V'$ of $\cS$. Let $U$ and $U'$ be objects of $\cT$. We have
natural isomorphisms
\begin{align*}
\Hom{\cT}(t'^{0}U,U') & \iso H^0(\cyl((s \tau'_{\leq -1} l U,\Sigma U') \to (U, \Sigma U'))) \\
& \iso H^0(\cyl((\tau'_{\leq -1}  l U, \Sigma rU') \to (U, \Sigma U'))) \\
& \iso H^0(\cyl((\tau'_{\leq -1} l U, \tau_{\leq -1} \Sigma r U') \to (U, \Sigma U'))).
\end{align*}
Here the last isomorphism holds because we have isomorphisms
\[
\Hom{\cS}(\tau'_{\leq -1} l U, \tau_{\leq_{-1}} \Sigma rU) \iso
\Hom{\cS}(\tau'_{\leq -1} l U, \Sigma rU)
\]
and
\[
\Hom{\cS}(\Sigma \tau'_{\leq -1} l U, \tau_{\leq_{-1}} \Sigma rU) \iso
\Hom{\cS}(\Sigma \tau'_{\leq -1} l U, \Sigma rU).
\]
We continue the chain of isomorphisms with
\begin{align*}
H^0(\cyl((\tau'_{\leq -1} l U, \tau_{\leq -1} \Sigma r U') \to (U, \Sigma U'))) & \iso
H^0(\cyl((lU, \Sigma \tau_{\leq 0} r U') \to (U,\Sigma U'))) \\
& \iso H^0(\cyl(U, \Sigma  s \tau_{\leq 0} r U') \to (U, \Sigma U'))) \\
& \iso \Hom{\cT}(U, t^{0} U').
\end{align*}
This shows the claim.
\end{proof}

\bibliographystyle{alpha}
\bibliography{plamondonYurikusa_keller.bib}

\end{document}